\documentclass[11pt]{article}
\pdfoutput=1

\usepackage[utf8]{inputenc}
\usepackage[T1]{fontenc}
\usepackage[letterpaper,margin=1in]{geometry}
\usepackage{microtype}
\usepackage[usenames,dvipsnames]{xcolor}
\usepackage{amsmath,amssymb,amsthm,mathtools}
\usepackage{enumitem}
\usepackage[numbers]{natbib}
\usepackage{hyperref}
\usepackage[nameinlink,noabbrev]{cleveref}

\hypersetup{
  pdftitle={A Proof of the Matrix Spencer Conjecture},
  pdfauthor={Emrullah Akbas, Suvrit Sra},
  pdfsubject={},
  pdfkeywords={matrix Spencer, discrepancy, small ball, Schur complement, partial coloring}
}

\usepackage{ss}

\setlist[itemize]{topsep=0.35em,itemsep=0.2em,leftmargin=2em}
\setlist[enumerate]{topsep=0.35em,itemsep=0.2em,leftmargin=2.2em}

\newcommand{\R}{\mathbb R}
\newcommand{\E}{\mathbb E}
\newcommand{\Prob}{\mathbb P}
\newcommand{\Tr}{\operatorname{Tr}}

\newcommand{\Sym}{\operatorname{Sym}}
\newcommand{\Cov}{\operatorname{Cov}}
\newcommand{\Var}{\operatorname{Var}}
\newcommand{\diag}{\operatorname{diag}}

\newcommand{\vecop}{\operatorname{vec}}

\newcommand{\one}{\mathbf 1}
\newcommand{\nlsum}{\sum\nolimits}
\newcommand{\nlprod}{\prod\nolimits}
\newcommand{\opnorm}[2][]{#1\lVert #2#1\rVert_{\mathrm{op}}}
\newcommand{\hsnorm}[2][]{#1\lVert #2#1\rVert_{F}}
\newcommand{\norm}[2][]{#1\lVert #2#1\rVert}
\newcommand{\ip}[3][]{#1\langle #2,#3#1\rangle}

\newtheorem{theorem}{Theorem}[section]
\newtheorem{proposition}[theorem]{Proposition}
\newtheorem{lemma}[theorem]{Lemma}

\theoremstyle{definition}

\numberwithin{equation}{section}

\title{A Proof of the Matrix Spencer Conjecture}
\author{%
\name Emrullah Akbas \email{emrullah.akbas@tum.de}\\
\addr Technical University of Munich, Garching, Germany\\[2pt]
\name Suvrit Sra \email{s.sra@tum.de}\\
\addr Technical University of Munich, Garching, Germany\\[2pt]
}

\begin{document}
\maketitle

\begin{abstract}
  We develop a novel approach to matrix discrepancy based on matrix small-ball estimates. Specifically, we use a determinantal weight (obtained from the log-barrier) to scale the small-ball probability into a partition function of a tilt of the Gaussian measure. We then employ matrix-weighted Poincar\'e inequalities to compare this partition function to that of a pinched or diagonal part of the matrix, obtaining \emph{dimension-free} constants. Our technique yields a hereditary small-ball estimate for Gaussian series that should be of independent interest. As the main application, we resolve the Matrix Spencer conjecture: for symmetric $n\times n$ matrices $A_1,\dots,A_n$ with $\opnorm{A_i}\le1$, one can efficiently find a coloring $x\in\{\pm1\}^n$ with $\opnorm[\big]{\nlsum_{i=1}^n x_iA_i}=O(\sqrt n)$.
\end{abstract}

\tableofcontents 

\section{Introduction}
In this work we study discrepancy minimization in the matrix setting. Recall the classical discrepancy setting where we are given vectors $a_1,\dots,a_n\in\R^d$ satisfying $\norm{a_i}_\infty\le1$. The goal is to assign a coloring $x_1,\dots, x_n\in\{\pm1\}$ to vectors to minimize the discrepancy $\norm[\big]{\nlsum_{i=1}^nx_ia_i}_\infty$. Choosing the coloring at random, a union bound argument shows that $\norm[\big]{\sum_{i=1}^nx_ia_i}_\infty\le O(\sqrt{n\log d})$. However, in a seminal result \citet{spencer1985} showed that the bound obtained via a random coloring can be beaten, which is in sharp contrast to other applications of the union bound in combinatorics where the probabilistic method is optimal. 
\begin{theorem}[Spencer \cite{spencer1985}]
    For all vectors $a_1,\dots,a_n\in\R^d$ satisfying, $\norm{a_i}_\infty\le1$, there exists $x_1,\dots, x_n\in\{\pm1\}$ such that $\norm[\big]{\nlsum_{i=1}^nx_ia_i}_\infty\le O(\sqrt{n}\max\{1,\sqrt{\log(d/n)}\})$.
\end{theorem}

In this work we study the matrix discrepancy, a natural generalization of Spencer's setting to matrices~\citep{zouzias2012, meka2014}. Let $A_1,\dots,A_n$ be $d\times d$ symmetric matrices having spectral norms $\opnorm{A_i}\le1$. The goal of matrix discrepancy is to find a coloring $x_1,\dots, x_n$ to minimize the spectral norm $\opnorm[\big]{\sum_{i=1}^nx_iA_i}$. Similarly to the classical setting,  matrix concentration inequalities show that randomly choosing $x$ gives a signing of discrepancy
$\opnorm[\big]{\nlsum_{i=1}^nx_iA_i}\le O(\sqrt{n\log d})$ \cite{AhlswedeWinter2002}. Mirroring the classical setting, it is a natural question to ask, whether it is possible to improve upon the bound obtained via a random coloring.


\subsection{Summary of Results}
The main result of this paper is a proof of the Matrix Spencer conjecture~\citep{zouzias2012, meka2014}.

\begin{theorem}[Matrix Spencer]
\label{thm:matrix-spencer}
There exists a universal constant $C>0$ such that for every set of symmetric $n \times n$ matrices $A_1,\dots,A_n$ with $\opnorm{A_i}\le1$
\[
\min_{x\in\{\pm1\}^n}\opnorm[\Big]{\nlsum_{i=1}^nx_iA_i}
\le
C\sqrt n.
\]
Moreover, a coloring with discrepancy at most $C\sqrt{n}$ can be found efficiently.
\end{theorem}

The proof of Theorem \ref{thm:matrix-spencer} proceeds via iteratively applying Gluskin's method for partial coloring \cite{Gluskin1989Extremal, Giannopoulos1997VectorBalancing}. We show the required Gaussian measure lower bounds in Theorem \ref{thm:mesoscopic-smallball}, which we refer to as
(hereditary) matrix small-ball. 

Compared to prior work \cite{dadush2022matrix,hopkins2022matrix,BansalJiangMeka2024MatrixSpencer}, we establish these lower bounds via a different route. We soften the indicator of the operator-norm ball by scaling it via a determinantal weight (inspired by log-barriers), which turns the small-ball probability into the partition function of a convex tilt of the Gaussian measure. The heart of the paper is a \emph{dimension-free ratio bound}: this partition function is bounded below by that of a suitably pinched version of the matrix, with constants independent of the matrix dimension (Theorem~\ref{thm:rsc} for a single rank-one pinching, Section~\ref{sec:interpolation} for the passage to the diagonal part). Since the partition function of a diagonal matrix is a scalar Gaussian quantity, such ratio bounds directly yield small-ball estimates. The key tool behind them is a matrix-weighted Poincar\'e inequality for the tilted measures, developed in Section~\ref{sec:matrix-poincare}. Our main small-ball result, which is what makes the partial coloring method succeed, is the following.

\begin{theorem}[Hereditary Matrix Small Ball]
\label{thm:mesoscopic-smallball}
Let $n\le d$. Then, for every $c_*>0$ there is a constant
$\kappa=\kappa(c_*)>0$ such that for all $d \times d$ symmetric matrices $A_1,\dots,A_n$ with $\opnorm{A_i}\le1$, we have

\[
\Prob\biggl\{
\opnorm[\Big]{\nlsum_{i=1}^n g_iA_i}
\le
\kappa\sqrt n
\Bigl(1+\log\frac{2d}{n}\Bigr)^2
\biggr\}
\ge
e^{-c_*n}.
\]
\end{theorem}

Theorem~\ref{thm:mesoscopic-smallball} is the hereditary form of the small-ball estimate. Its non-hereditary counterpart, obtained by iterating the one-step ratio bound of Theorem~\ref{thm:rsc} along $d-1$ rank-one pinchings, is the following small-ball estimate with explicit constants. It is also the starting point of our approach, 
which the interpolation argument of Section~\ref{sec:hereditary-small-ball} then generalizes.

\begin{theorem}[Matrix Small Ball] \label{thm:matrixsmall}
    Let $X = \nlsum_{i=1}^n g_i A_i$ be a $d \times d$ symmetric random Gaussian matrix with $\E X^2 = \nlsum_{i=1}^n A_i^2 \preceq I_d $. Then there exist universal constants $C,c>0$ such that $$\Prob\bigl\{\opnorm{X} < C \bigr\} = \Prob\Bigl\{\opnorm[\Big]{\nlsum_{i=1}^n g_i A_i} < C \Bigr\} \geq \exp(-cd).$$
    In particular, the absolute constants can be chosen to satisfy $C = 4$ and 
$c\approx 2.614861$.
\end{theorem}

\subsection{Discussion and Related Work}
The Matrix Spencer problem dates back to \cite{zouzias2012,meka2014}. We also refer  the reader to the 
nice blog post by Afonso Bandeira \cite{bandeira2025randomstrasse101openproblems2024} for a detailed 
exposition.  A substantial line of work has studied various approaches to analyze  Matrix Spencer.   
Hopkins, Raghavendra and Shetty \cite{hopkins2022matrix} established the conjectured $O(\sqrt n)$ bound 
for  matrices of rank at most $\sqrt{n}$, using tools from quantum communication complexity and quantum information theory. Around the same time, Dadush, Jiang and Reis \cite{dadush2022matrix}
reduced Matrix Spencer to constructing quantum relative entropy nets for the spectraplex, proving low-
rank and block-diagonal matrix discrepancy bounds.  Bansal, Jiang and Meka  
\cite{BansalJiangMeka2024MatrixSpencer} resolved Matrix Spencer for matrices up to poly-logarithmic 
rank $n/\log^3 n$. In their proof they established a close connection of matrix discrepancy to the 
matrix concentration inequalities developed in \cite{BandeiraBoedihardjoVanHandel2023}. Moreover,  
recent work by Bandeira and B\"olcskei \cite{BandeiraBolcskei2026MatrixDiscrepancy}, and  Akbas and Sra 
\cite{akbas2026algebraicmatrixspencertheorem}
established Matrix Spencer for families of matrices that correspond to the regular representation of a finite group, or more generally, that are elements of a finite dimensional $C^*$-algebra of dimension at most $O(n)$.

This work studies a novel approach to matrix discrepancy. We reduce Matrix Spencer to lower bounds on partition functions of convex tilts of the Gaussian measure, and we obtain these from dimension-free bounds on the ratio between the partition function of a matrix series and that of its pinched or diagonal part, as outlined above.

\section{Proof of the Matrix Small Ball Theorem~\ref{thm:matrixsmall}}
\label{sec:matrixsmall}
The idea for proving the Matrix Small Ball Theorem may intuitively be viewed as using $\det(I-X^2)$ as a proxy for the spectral norm. That is, we define the partition function
\begin{equation}\label{eq:partition-function}
\mathcal Z_{R,\alpha}(B) :=
\E[
\det(I_d-B(g)^2)^\alpha
\one_{\{\opnorm{B(g)}<1\}}
],
\end{equation}
where $B(g)=\frac1R\nlsum_i g_iA_i$ for parameters $R>0$ and integer $\alpha\ge1$. Since the integrand in \eqref{eq:partition-function} lies in $[0,1]$, we immediately get the bound
\begin{equation}\label{eq:partition-below-prob}
\mathcal Z_{R,\alpha}(B)
\le
\Prob\Bigl\{\opnorm[\Big]{\nlsum_i g_iA_i}<R\Bigr\}.
\end{equation}
Therefore, based on \eqref{eq:partition-below-prob} the proof of Theorem \ref{thm:matrixsmall} reduces to proving lower bounds on the partition function in \eqref{eq:partition-function}. The determinantal weight is what makes this tractable: it is log-concave, it vanishes exactly at the boundary of the operator-norm ball, and it factorizes exactly under rank-one pinching (Lemma~\ref{lem:schur-factorization} below). To this end, let $v$ be some unit vector and consider the block decomposition of the $A_i$ with respect to $\R^d = \text{span}\{v\}\oplus v^\top$,
\[
A_i=
\begin{pmatrix}
a_i&u_i^\top\\
u_i&D_i
\end{pmatrix}.
\]
Now write $a=\frac1R\nlsum_i g_i a_i$, $b=\frac1R\nlsum_i g_i u_i$, and $C=\frac1R\nlsum_i g_iD_i$, and then define the block matrices
\begin{equation}\label{eq:full-pinched-blocks}
B=
\begin{pmatrix}
a&b^\top\\b&C
\end{pmatrix},
\qquad
B_0=
\begin{pmatrix}
a&0\\0&C
\end{pmatrix},
\end{equation}
omitting the dependence on $g$ for readability. Now, let $\Omega_0=\{g:|a|<1,\ \opnorm C<1\}$ and define on this event $\Omega_0$
\begin{equation}\label{eq:xi-definitions}
\xi_+=\frac{b^\top(I+C)^{-1}b}{1+a},
\qquad
\xi_-=\frac{b^\top(I-C)^{-1}b}{1-a}.
\end{equation}

\begin{lemma}[Schur-complement factorization]\label{lem:schur-factorization}
On the event $\Omega_0$ we have that,
\begin{align}
\det(I+B)&=\det(I+C)(1+a)(1-\xi_+),\label{eq:schur-plus}\\
\det(I-B)&=\det(I-C)(1-a)(1-\xi_-).\label{eq:schur-minus}
\end{align}
Consequently,
\begin{equation}\label{eq:schur-product}
\det(I-B^2)
=
\det(I-B_0^2)(1-\xi_+)(1-\xi_-)
\end{equation}
and,
\begin{equation}\label{eq:equivalence}
    \opnorm B<1
\quad\Longleftrightarrow\quad
g\in\Omega_0,\ \xi_+<1,\ \xi_-<1.
\end{equation}
\end{lemma}

\begin{proof}
Applying the Schur-complement formula to $I+B$ with lower block $I+C$, and to $I-B$ with lower block $I-C$ gives us \eqref{eq:schur-plus} and \eqref{eq:schur-minus}.  Moreover, positive definiteness of both $I+B$ and $I-B$ gives the final equivalence in \eqref{eq:equivalence}
\end{proof}

We define
\begin{equation} \label{eq:w0}
    w_0(g)=
(1-a(g)^2)^\alpha
\det(I-C(g)^2)^\alpha
\one_{\Omega_0}(g),
\qquad
 d\mu_0(g)=\frac{w_0(g)}{\E w_0}\,d\gamma_n(g),
\end{equation}
and we denote with $\mathcal E_P(M)=PMP+QMQ$ the pinching map associated with the orthogonal projection $P= vv^\top$. From now on we fix $R=\alpha=4$ and let
\[
B(g)=\frac14\nlsum_i g_iA_i
=
\begin{pmatrix}
a&b^\top\\
b&C
\end{pmatrix},
\qquad
\mathcal E_P(B)=B_0(g)=
\begin{pmatrix}
a&0\\
0&C
\end{pmatrix},
\]
and denote with
\begin{equation}
    \label{eq:zfull}
    Z_{\mathrm{full}} = \mathcal Z_{R,\alpha}(B)
=
\E\bigl[
\det(I_d-B(g)^2)^4
\one_{\{\opnorm{B(g)}<1\}}
\bigr]
\end{equation}
and
\begin{equation}
    \label{eq:zpin}
    Z_{\mathrm{pin}} = \mathcal Z_{R,\alpha}(\mathcal E_P(B))
=
\E\bigl[
\det(I_d-B_0(g)^2)^4
\one_{\{\opnorm{B_0(g)}<1\}}
\bigr],
\end{equation}
the respective partition functions. Equation \eqref{eq:schur-product} in Lemma \ref{lem:schur-factorization} is then equivalent to the following identity
\begin{equation}\label{eq:one-step-ratio}
\frac{Z_{\mathrm{full}}}{Z_{\mathrm{pin}}}=
\frac{\mathcal Z_{R,\alpha}(B)}
{\mathcal Z_{R,\alpha}(\mathcal E_P(B))}
=
\E_{\mu_0}
(1-\xi_+)_{+}^{\alpha}
(1-\xi_-)_{+}^{\alpha}.
\end{equation}
Hence, our goal will be proving lower bounds on the right-hand side of equation 
\eqref{eq:one-step-ratio}. Indeed, in   Section \ref{sec:dim-free} we will prove a dimension-free lower bound $Z_{\mathrm{full}} \geq c_\text{Schur} Z_{\mathrm{pin}}$ for the right-side of \eqref{eq:one-step-ratio} for $R=\alpha=4$. At a high level, the proof of Theorem \ref{thm:matrixsmall} proceeds by iterating this procedure $d-1$ times, until we arrive at a diagonal matrix at the final stage, at which point we simply apply the Gaussian correlation inequality.

\subsection{Matrix-Weighted Gaussian Poincar\'e Inequality }
\label{sec:matrix-poincare}
The goal of this section is to develop tools from high-dimensional probability and the analysis of Gaussian measures that we will use in Section \ref{sec:dim-free} to prove the aforementioned dimension-free lower bound for ratios of partition
functions.  In particular, we will need the following matrix-valued spectral gap estimate.

\begin{lemma}[Matrix-weighted Gaussian Poincar\'e inequality]
\label{lem:cofactor-poincare}
Let $H_0,E_1,\dots,E_n$ be symmetric $r\times r$ matrices, and set
\[
H(x)=H_0+\nlsum_{i=1}^n x_iE_i,
\qquad
\Omega=\{x\in\R^n:H(x)\succ0\},
\]
which we assume to be nonempty. Define the matrix weight
\[
\mathcal K(x)
=
\det(H(x))^4H(x)^{-1}\one_\Omega(x).
\]
Then, for every $F\in C^1(\Omega;\R^r)$ with
$\int F^\top\mathcal KF\,d\gamma_n<\infty$ (in particular, for every
polynomial map $F$; the right-hand side below may be $+\infty$),
\begin{align}
&\int F^\top\mathcal K F\,d\gamma_n
-
\Bigl(\int\mathcal K F\,d\gamma_n\Bigr)^\top
\Bigl(\int\mathcal K\,d\gamma_n\Bigr)^{-1}
\Bigl(\int\mathcal K F\,d\gamma_n\Bigr)
\nonumber\\
&\hspace{35mm}
\le
\int\nlsum_{i=1}^n
(\partial_iF)^\top\mathcal K(\partial_iF)\,d\gamma_n.
\label{eq:cofactor-poincare}
\end{align}
The same conclusion holds after multiplying $\mathcal K$ by a $C^\infty$
log-concave scalar factor.
\end{lemma}

Lemma~\ref{lem:cofactor-poincare} is a special case of the following
matrix-weighted analogue of the classical curvature criterion for
Poincar\'e inequalities. In the scalar setting, a uniform lower
curvature bound yields a Poincar\'e inequality via the Bochner
argument; the lemma below shows that the same mechanism extends to
matrix-valued weights, with the scalar curvature bound replaced by the
corresponding block-matrix inequality. Its proof, given in
Appendix~\ref{app:curvature-criterion}, proceeds on smooth bounded
convex subdomains compactly contained in $\Omega$, with Neumann boundary
conditions, and then exhausts $\Omega$. In this way no assumption on the
behaviour of the weight near $\partial\Omega$ is needed, and no
restriction on the support of the test function.

\begin{lemma}[Matrix-weighted curvature criterion]
\label{lem:matrix-weighted-curvature}
Let $\Omega\subseteq\R^n$ be a nonempty open convex set, and let
$W:\Omega\to\Sym_r^{++}(\R)$ be a $C^\infty$ matrix weight such that
$\int_\Omega W(x)\,dx$ is finite (and hence positive definite). Define
\[
\Theta_{ij}(W)
=
(\partial_iW)W^{-1}(\partial_jW)-\partial_{ij}W,
\qquad 1\le i,j\le n.
\]
Suppose that there is some $\lambda>0$ such that for all $x\in\Omega$ and
all $z_1,\dots,z_n\in\R^r$ there holds,
\begin{equation}
\nlsum_{i,j=1}^n
z_i^\top\Theta_{ij}(W)z_j
\ge
\lambda
\nlsum_{i=1}^n z_i^\top Wz_i.
\label{eq:abstract-curvature-gap}
\end{equation}
Then, every $F\in C^1(\Omega;\R^r)$ with $\int_\Omega F^\top WF\,dx<\infty$
satisfies
\begin{align}
&\int_\Omega F^\top WF\,dx
-
\Bigl(\int_\Omega WF\,dx\Bigr)^\top
\Bigl(\int_\Omega W\,dx\Bigr)^{-1}
\Bigl(\int_\Omega WF\,dx\Bigr)
\nonumber\\
&\hspace{25mm}
\le
\lambda^{-1}
\int_\Omega
\nlsum_{i=1}^n
(\partial_iF)^\top W(\partial_iF)\,dx,
\label{eq:abstract-matrix-poincare}
\end{align}
where the right-hand side may be $+\infty$. (Note that $\int_\Omega WF\,dx$
converges absolutely, since
$\norm{WF}_2\le\norm{W}_{\mathrm{op}}^{1/2}(F^\top WF)^{1/2}$.)
\end{lemma}

\subsection{Dimension-free Lower Bound for the Ratio of Partition Functions}
\label{sec:dim-free}
In this section, we establish the following dimension-free lower bound for ratios of partition functions of the form given in \eqref{eq:one-step-ratio}. Then in Section \ref{sec:proofsmallball} we use it to deduce Theorem \ref{thm:matrixsmall}.

\begin{theorem}[Dimension-free lower bounds on ratio of partition functions]\label{thm:rsc} Let $A_1,\dots,A_n$ be symmetric $d \times d$ matrices
satisfying $\nlsum_{i=1}^nA_i^2\preceq I_d.$
Fix a unit vector $v\in\R^d$ and consider the block-decomposition of the matrices $A_i$ with respect to $\R^d=\operatorname{span}\{v\}\oplus v^\perp$
\[
A_i=
\begin{pmatrix}
a_i&u_i^\top\\
u_i&D_i
\end{pmatrix}.
\]
Let $g\in\R^n$ and denote $a(g)=\frac14\nlsum_i g_ia_i$, $b(g)=\frac14\nlsum_i g_iu_i$, and $C(g)=\frac14\nlsum_i g_iD_i$. Then, consider the tilted Gaussian measure 
\[
d\mu_0(g)=\frac{w_0(g)}{\E w_0}\,d\gamma_n(g),\quad\text{where}\quad
w_0(g)=
(1-a(g)^2)^4
\det(I_{d-1}-C(g)^2)^4
\one_{\Omega_0}(g),
\]
and the set $\Omega_0=\{g:|a(g)|<1,\ \opnorm{C(g)}<1\}$. Moreover, for $g\in\Omega_0$, define
\begin{equation}\label{eq:schur-psi}
\xi_+(g)=\frac{b(g)^\top(I+C(g))^{-1}b(g)}{1+a(g)},
\qquad
\xi_-(g)=\frac{b(g)^\top(I-C(g))^{-1}b(g)}{1-a(g)}.
\end{equation}
Then, we have that
\begin{equation}\label{eq:RSC}
\E_{\mu_0}
(1-\xi_+)_{+}^{4}(1-\xi_-)_{+}^{4}
\ge c_{\mathrm{Schur}},
\end{equation}
where
\begin{equation}\label{eq:rsc-constant}
c_{\mathrm{Schur}}
=
\Bigl[
1-\frac18(\Lambda^2+\Lambda^{-2})
\Bigr]^8,
\qquad
\Lambda=8-4\sqrt3.
\end{equation}
\end{theorem}
Before we prove Theorem~\ref{thm:rsc} we need to develop some supporting material. Toward that end, the main idea is to apply the matrix-weighted Poincar\'e inequality of Section \ref{sec:matrix-poincare} to  quantities arising from the block decomposition of $B(g) = \nlsum_i g_iA_i$. We begin by introducing the quantities that will consider in the following. It will be convenient to absorb the factor 1/4 into the coefficient matrices. Thus, with a slight abuse of notation, we write
\[
B(g)=\nlsum_i g_iA_i
=
\begin{pmatrix}
a(g)&b(g)^\top\\
b(g)&C(g)
\end{pmatrix},
\]
and
\[
A_i=
\begin{pmatrix}
a_i&u_i^\top\\
u_i&D_i
\end{pmatrix},
\qquad
U=\nlsum_i u_iu_i^\top,
\qquad
\sigma=\Tr U=\nlsum_i\norm{u_i}^2.
\]
Since $\nlsum_iA_i^2\preceq\frac1{16}I_d$, we moreover have 
\begin{equation}\label{eq:scaled-top-variance}
\nlsum_i a_i^2+\sigma\le\frac1{16}, \qquad U+\nlsum_iD_i^2\preceq\frac1{16}I_{d-1}, \qquad 0\le\sigma\le\frac1{16}.
\end{equation}

The remainder of the argument is organized around the lower-right block
$C$. The Schur-complement terms arising from the off-diagonal block $b$
involve the resolvent $(I-C^2)^{-1}$, averaged under a Gaussian measure tilted  by $\det(I-C(g)^2)^4
\one_{\{\opnorm{C(g)}<1\}}$. We therefore first isolate this
measure and establish a dimension-free bound on the corresponding
averaged resolvent. More precisely, we consider the tilted Gaussian measure 
\begin{equation}
d\nu(g)
=
\frac{
\det(I-C(g)^2)^4
\one_{\{\opnorm{C(g)}<1\}}
}{Z_C}
\,d\gamma_n(g),
\label{eq:C-only-measure}
\end{equation}
where $Z_C
=
\E\bigl[
\det(I-C(g)^2)^4
\one_{\{\opnorm{C(g)}<1\}}
\bigr]$
 denotes the partition function associated with the $C$-block. Under this measure, we set
\begin{equation} \label{eq:def-S}
    G(g)=(I-C(g)^2)^{-1},
\qquad
S=\E_\nu G.
\end{equation}
Thus $S$ is the averaged resolvent of the lower-right block under this
tilted Gaussian measure. Since $G\succeq I$ holds pointwise, we immediately
have $S\succeq I.$ The key point is that $S$ also admits a dimension-free upper bound. The matrix-weighted Poincaré inequality  will precisely allow us to derive such a dimension-free upper bound: with a suitable test function, its weighted variance reduces to $I-S^{-1}$, while its Dirichlet term can be bounded using the variance normalization on the coefficient matrices. This will then yield a closed matrix inequality for $S$.
\begin{proposition}[Dimension-free averaged resolvent]
\label{prop:S-bound}
Let $S$ be defined as in \eqref{eq:def-S}, then we have 
\[
I\preceq S\preceq\Lambda I,
\qquad
\Lambda=8-4\sqrt3.
\]
\end{proposition}

We next turn to the contribution of the off-diagonal block $b$.
Recall that the two quadratic terms arising from the Schur-complement
representation in \eqref{eq:schur-psi} involve the resolvents
$(I+C)^{-1}$ and $(I-C)^{-1}$. Now, since $(I+C)^{-1}+(I-C)^{-1}
=
2(I-C^2)^{-1}$,
it is natural to introduce
\begin{equation}
q(g)
=
b(g)^\top G(g)b(g),
\qquad
G(g)=(I-C(g)^2)^{-1}.
\label{eq:q-definition}
\end{equation}
Controlling $\E_\nu q$ will therefore yield a simultaneous bound on the
two Schur-complement contributions appearing in
\eqref{eq:schur-psi}. Note, that Proposition~\ref{prop:S-bound} provides a uniform bound on the averaged
resolvent $S=\E_\nu G$, but this does not immediately yield a bound on
$\E_\nu[b^\top G b]$, since $b$ and $G$ are correlated. The next proposition controls this dependence. 

The natural scale for such an estimate is $\sigma=\nlsum_{i=1}^n\norm{u_i}^2=\Tr U$.  Indeed, since $b(g)=\nlsum_i g_i u_i$, we have $\E_{\gamma_n}\norm{b(g)}_2^2=\sigma$, so that $\sigma$ corresponds the Gaussian second-moment scale of the
off-diagonal block. Moreover, as the proof below shows, $\sigma$ also arises naturally from the Dirichlet term in the matrix-weighted Poincaré inequality.

\begin{proposition}[Resolvent-weighted quadratic-form bound]
\label{prop:q-bound}
Let $q$ be defined as in equation \eqref{eq:q-definition} and let $\sigma=\nlsum_{i=1}^n\norm{u_i}^2$. Then, we have
\begin{equation}
\E_\nu q
\le
\Lambda^2\sigma,
\label{eq:q-bound}
\end{equation}
where $\Lambda=8-4\sqrt3.$
\end{proposition}

Finally, we can now combine the preceding estimates to complete the proof of the dimension-free lower bound in Theorem~\ref{thm:rsc}.

\begin{proof}[Proof of Theorem~\ref{thm:rsc}]
Let $Z_{\mathrm{full}}$ and $Z_{\mathrm{pin}}$ be defined as in
\eqref{eq:zfull} and \eqref{eq:zpin}, respectively. We also denote the
partition function associated with the $C$-block by,
\begin{equation*}
Z_C
=
\E\bigl[
\det(I-C(g)^2)^4
\one_{\{\opnorm{C(g)}<1\}}
\bigr].
\end{equation*}
 Moreover, let $\opnorm{C}<1$ and set
\[
p_+
=
b^\top(I+C)^{-1}b,
\qquad
p_-
=
b^\top(I-C)^{-1}b, \qquad q=b^\top(I-C^2)^{-1}b.
\]

Since $(I+C)^{-1}+(I-C)^{-1}
=
2(I-C^2)^{-1}$, we have that $p_++p_-=2q$. Now, using Lemma~\ref{lem:schur-factorization}, we obtain
\begin{equation}
\frac{Z_{\mathrm{full}}}{Z_C}
=
\E_\nu
[
(1+a-p_+)_{+}^{4}
(1-a-p_-)_{+}^{4}
].
\label{eq:full-over-C}
\end{equation}
 Recall that proving \eqref{eq:RSC} is equivalent to showing that $\frac{Z_{\mathrm{full}}}{Z_\mathrm{pin}}\geq c_\text{Schur}$. However, since $(1-a^2)^4\one_{\{|a|<1\}}\le1$, we have that $Z_{\mathrm{pin}}\le Z_C$, and 
it therefore suffices to obtain a uniform lower bound on the right-hand
side of \eqref{eq:full-over-C}.

To this end, we first show that
\begin{equation}
(1+a-p_+)_{+}^{4}(1-a-p_-)_{+}^{4}
\ge
(1-|a|-p_+-p_-)_{+}^{8}.
\label{eq:two-slack-lower}
\end{equation}
Indeed, whenever the expression on the right-hand side of \eqref{eq:two-slack-lower} is nonzero,
\[
1-|a|-p_+-p_-
\le
1+a-p_+,
\qquad
1-|a|-p_+-p_-
\le
1-a-p_-,
\]
whereas if it is zero the inequality is immediate. Moreover, using $p_++p_-=2q$ we get that,
\[
(1+a-p_+)_{+}^{4}(1-a-p_-)_{+}^{4}
\ge
(1-|a|-2q)_+^8.
\]

Since $t\longmapsto(1-t)_+^8$ is convex, Jensen's inequality together with
\eqref{eq:full-over-C} yields
\begin{equation}
\frac{Z_{\mathrm{full}}}{Z_C}
\ge
(
1-\E_\nu|a|-2\E_\nu q
)_+^8.
\label{eq:RSC-Jensen}
\end{equation}
Thus, proving a uniform lower bound on the right-hand
side of \eqref{eq:full-over-C} reduces to controlling the two scalar
quantities $\E_\nu|a|$ and $\E_\nu q$.
For the latter, Proposition~\ref{prop:q-bound} shows that,
\begin{equation}
\E_\nu q\le\Lambda^2\sigma.
\label{eq:q-bound-use}
\end{equation}
It then remains to control $\E_\nu|a|$. The measure $\nu$ is even and
$1$-strongly log-concave on the convex set
$\{\opnorm{C}<1\}$. Indeed, its potential is given by
\[
V(g)
=
\frac12\norm{g}^2
-4\log\det(I+C(g))
-4\log\det(I-C(g)).
\]
Therefore, using that $\nu$ is $1$-strongly log-concave, the Brascamp--Lieb variance inequality~\cite{BrascampLieb1976} yields
\[
\Cov_\nu(g)\preceq I_n.
\]
Moreover,  since $\nu$ is even, we have $\E_\nu g=0$, and therefore
$\E_\nu a^2
=
\Var_\nu(a)
\le
\nlsum_i a_i^2.$
Furthermore, using \eqref{eq:scaled-top-variance}, we obtain that
\begin{equation}
\E_\nu|a|
\le
\sqrt{\E_\nu a^2}
\le
\Bigl(\nlsum_i a_i^2\Bigr)^{1/2}
\le
\sqrt{\frac1{16}-\sigma}.
\label{eq:a-bound}
\end{equation}

Finally, combining \eqref{eq:RSC-Jensen}, \eqref{eq:q-bound-use}, and
\eqref{eq:a-bound}, and using $Z_{\mathrm{pin}}\le Z_C$ , gives
\begin{equation*}
\frac{Z_{\mathrm{full}}}{Z_{\mathrm{pin}}}
\ge
\Bigl(
1-\sqrt{\frac1{16}-\sigma}
-2\Lambda^2\sigma
\Bigr)_+^8.
\end{equation*}

It remains only to optimize the scalar expression on the right. Define
\[
f(s)
=
1-\sqrt{\frac1{16}-s}-2\Lambda^2s,
\qquad
0\le s\le\frac1{16}.
\]
Since $f''(s)
=
\frac{1}{4(1/16-s)^{3/2}}
>0$
the function $f$ is strictly convex. Solving $f'(s)=0$ gives $\sqrt{\frac1{16}-s}
=
\frac{1}{4\Lambda^2}.$
Substituting this identity into $f$ yields $\min_{0\le s\le1/16}f(s)
=
1-\frac18(\Lambda^2+\Lambda^{-2})$, so that

\[
\frac{Z_{\mathrm{full}}}{Z_{\mathrm{pin}}}
\ge
\Bigl[
1-\frac18(\Lambda^2+\Lambda^{-2})
\Bigr]^8
=
c_{\mathrm{Schur}},
\]
which proves Theorem~\ref{thm:rsc}.
\end{proof}

\subsection{Proof of Theorem~\ref{thm:matrixsmall}} \label{sec:proofsmallball}
With Theorem~\ref{thm:rsc} in hand, the matrix small-ball Theorem follows
by iterating the one-step partition-function comparison along successive
rank-one pinchings. Each application reduces the matrix dimension by one
while incurring only a universal constant-factor loss. After $d-1$
iterations, the remaining matrix is diagonal, and the resulting scalar
Gaussian small-ball probability is bounded below using the \v{S}id\'ak
inequality \cite{Sidak1968}.

\begin{proof}[Proof of Theorem~\ref{thm:matrixsmall}]
Fix an orthonormal basis $e_1,\dots,e_d$ of $\R^d$, and successively
apply the pinching maps $\mathcal E_{P_k}(A)
=
P_kAP_k+(I-P_k)A(I-P_k)$ associated with the corresponding rank-one projections $P_k=e_ke_k^\top$
for $k=1,\dots,d-1$. We first note that the variance constraint is
preserved under each such pinching. \footnote{Indeed, if $\mathcal E_P(A)=PAP+(I-P)A(I-P),$
then, for every symmetric matrix $A$ we have that $\mathcal E_P(A)^2
\preceq
\mathcal E_P(A^2)$. To see this, writing $Q=I-P$ gives $\mathcal E_P(A^2)-\mathcal E_P(A)^2
=
PAQAP+QAPAQ
\succeq0.$} Consequently, if $\nlsum_iA_i^2\preceq I_d$, then also $\nlsum_i\mathcal E_P(A_i)^2
\preceq
\mathcal E_P\bigl(\nlsum_iA_i^2\bigr)
\preceq I_d.$
Therefore, the hypothesis of Theorem~\ref{thm:rsc} remains valid at every
stage of the iteration.

Now, applying Theorem~\ref{thm:rsc} successively along
$e_1,\dots,e_{d-1}$, each pinching incurs at most the universal factor
$c_{\mathrm{Schur}}$. After $d-1$ steps all off-diagonal entries have
been removed. Denoting the resulting diagonal matrix by
$A_i^{\mathrm{diag}}$, we therefore obtain
\begin{equation}
\mathcal Z_{4,4}(B)
\ge
c_{\mathrm{Schur}}^{\,d-1}
\mathcal Z_{4,4}(B_{\mathrm{diag}}),
\label{eq:iterated-rsc}
\end{equation}
where
\[
B_{\mathrm{diag}}(g)
=
\frac14\nlsum_i g_iA_i^{\mathrm{diag}}
=
\diag(Z_1,\dots,Z_d).
\]

It then remains to lower bound the partition function $\mathcal Z_{4,4}(B_{\mathrm{diag}})$. To this end, let $f(t)=(1-t^2)_+^4$ and $r(s)=\sqrt{1-s^{1/4}}$, then we obtain  
\[
\mathcal Z_{4,4}(B_{\mathrm{diag}})
=
\E\nlprod_{j=1}^d f(Z_j)= \int_{[0,1]^d}
\Prob\{
|Z_j|<r(s_j),\ 1\le j\le d
\}
\,ds_1\cdots ds_d, 
\]
where in the second equality we used the layer-cake representation and
Tonelli's theorem. Applying the \v{S}id\'ak inequality~\cite{Sidak1968} to the Gaussian
vector $(Z_1,\dots,Z_d)$ inside the integral, we obtain
\begin{align}
\mathcal Z_{4,4}(B_{\mathrm{diag}})
&\ge
\int_{[0,1]^d}
\nlprod_{j=1}^d
\Prob\{|Z_j|<r(s_j)\}
\,ds_1\cdots ds_d
\nonumber\\
&=
\nlprod_{j=1}^d
\int_0^1
\Prob\{|Z_j|<r(s)\}\,ds =\nlprod_{j=1}^d \E f(Z_j).
\label{eq:diagonal-sidak}
\end{align}

Now, since $\nlsum_i(A_i^{\mathrm{diag}})^2 \preceq I$, we have $\Var(Z_j)\le\frac1{16}$ for $ 1\le j\le d$. By using the elementary inequality $$(1-u)_+^4\ge1-4u, \qquad u\ge0,$$ we obtain $\E f(Z_j)
=
\E(1-Z_j^2)_+^4
\ge
1-4\E Z_j^2\ge \frac{3}{4}$, and therefore also
\begin{equation}
\mathcal Z_{4,4}(B_{\mathrm{diag}})
\ge
\Bigl(\frac34\Bigr)^d.
\label{eq:diagonal-partition-lower}
\end{equation}
Combining \eqref{eq:iterated-rsc} and
\eqref{eq:diagonal-partition-lower}, we conclude that
\begin{equation}
    \mathcal Z_{4,4}(B)
\ge
c_{\mathrm{Schur}}^{\,d-1}
\Bigl(\frac34\Bigr)^d.
\label{eq:schur-final}
\end{equation}
Finally, by \eqref{eq:partition-below-prob},
\[
\mathcal Z_{4,4}(B)
\le
\Prob\Bigl\{
\opnorm[\Big]{\nlsum_i g_iA_i}<4
\Bigr\},
\]
and therefore by using \eqref{eq:schur-final} we obtain,
\begin{equation} \label{eq:schur-final-almost}
    \Prob\Bigl\{
\opnorm[\Big]{\nlsum_i g_iA_i}<4
\Bigr\}
\ge
c_{\mathrm{Schur}}^{\,d-1}
\Bigl(\frac34\Bigr)^d.
\end{equation}
Since
\[
c_{\mathrm{Schur}}
=
\Bigl[
1-\frac18(\Lambda^2+\Lambda^{-2})
\Bigr]^8,
\qquad
\Lambda=8-4\sqrt3,
\]
 we then obtain for the right-hand side of \eqref{eq:schur-final-almost}
\[
\Prob\Bigl\{
\opnorm[\Big]{\nlsum_i g_iA_i}<4
\Bigr\}
\ge
e^{-cd},
\]
with
\[
c
=
-8\log\Bigl(\frac{1020\sqrt3-1671}{128}\Bigr)
-\log\Bigl(\frac34\Bigr) \approx 2.614861,
\]
which concludes the proof of the Matrix Small Ball Theorem.
\end{proof}

\section{Proof of the Hereditary Matrix Small Ball Theorem}
\label{sec:hereditary-small-ball}

The goal of this section is to prove Theorem~\ref{thm:mesoscopic-smallball}. The idea of the proof will be similar in spirit to the one we presented in Section \ref{sec:matrixsmall}. We will again reduce proving matrix small-ball estimates to proving lower bounds on certain  weighted partition functions supported on  operator-norm balls.  The mechanism for
obtaining  lower bound in the present setting, however, is different. Rather than iterating rank-one pinchings, we compare the full Gaussian matrix with its diagonal part through a continuous interpolation and control the resulting loss using the matrix-weighted Poincar\'e inequalities developed in Section \ref{sec:matrix-poincare}.

Let $A_1,\dots,A_n$ symmetric $d \times d$ matrices  with $\opnorm{A_i}\le1$ and let $L=1+\log\frac{2d}{n},
\,
R=\kappa\sqrt n\,L^2$.
It will be convenient to  rescale the matrices $A_i$ by setting $M_i=\frac{A_i}{R}$. Let $B(g)=\nlsum_{i=1}^n g_iM_i$.  Since $\norm{A_i}_{\text{op}}\le1$, we have
$\nlsum_{i=1}^nM_i^2
\preceq
\eta I_d$ and $\eta:=\frac1{\kappa^2L^4}.$
With this normalization, the small-ball event in
Theorem~\ref{thm:mesoscopic-smallball} becomes
\[
\Bigl\{\opnorm[\Big]{\nlsum_{i=1}^n g_iA_i} < R\Bigr\}
=\bigl\{\opnorm{B(g)}<1\bigr\}.
\]
Thus, it suffices to prove a lower bound of the order $e^{-c_*n}$ for the
Gaussian measure of this normalized operator-norm unit ball.

To this end, similarly as in Section \ref{sec:matrixsmall}, we will consider suitable smooth surrogate $e^{-\mathcal{V}(B)}$ for the indicator $\one_{\{\opnorm B<1\}}$, satisfying
\[
0\le e^{-\mathcal V(B)}
\le
\one_{\{\opnorm B<1\}},
\]
where $\mathcal V(B)=\Tr v(B)$ is an even convex spectral  potential function which will be specified below in Section \ref{sec:potential}. 
Its associated partition function
\begin{equation}
Z
=
\E_{\gamma_n}\![e^{-\mathcal V(B(g))}],
\label{eq:mesoscopic-partition}
\end{equation}
satisfies then
\[
Z
\le
\Prob\{\opnorm{B(g)}<1\}.
\]
As in Section \ref{sec:matrixsmall}, this simple yet crucial relation between partition functions and matrix
small-ball events is the starting point of our proof of
Theorem~\ref{thm:mesoscopic-smallball}: it reduces the desired Gaussian measure lower bound to establishing a lower bound on $Z$ of the form,
\[
Z\ge e^{-c n}.
\]

Our proof will follow a surprisingly simple strategy. We fix an orthonormal basis and decompose
\[
B(g)
=
B_{\mathrm d}(g)+B_{\mathrm o}(g),
\]
where $B_{\mathrm d}(g)$ is the diagonal part of $B(g)$ and
$B_{\mathrm o}(g)=B(g)-B_{\mathrm d}(g)$ contains its off-diagonal
entries. We interpolate between these two matrices by setting
\begin{equation}
B_t(g)
=
B_{\mathrm d}(g)+tB_{\mathrm o}(g),
\qquad
0\le t\le1.
\label{eq:mesoscopic-interpolation}
\end{equation}
Thus $B_0=B_{\mathrm d}$ is diagonal, whereas $B_1=B$ is the original
Gaussian matrix series.

Moreover, for each $t\in[0,1]$, we let

\begin{equation}\label{eq:zt}
    Z_t
=
\E_{\gamma_n}\![e^{-\mathcal V(B_t(g))}],
\qquad
F(t)=\log Z_t.
\end{equation}
The idea is to first obtain a direct lower bound on the diagonal
partition function $Z_0$, and then show that introducing the
off-diagonal part along the interpolation path does not decrease the
partition function substantially. The quality of this comparison is
governed by the second derivative of $F(t)=\log Z_t,$
so that the main task will be to establish a uniform lower bound on $F''(t)$.

More precisely, we first note that 
\[
F(1)-F(0)
=
\log\frac{Z_1}{Z_0}.
\]
Thus, a lower bound on $F(1)-F(0)$ translates directly into a
multiplicative comparison between the partition functions of the full
matrix and its diagonal part. However, by Taylor's formula we have that,
\[
F(1)
=
F(0)+F'(0)+\int_0^1(1-t)F''(t)\,dt.
\]
Therefore, having good control on first and second order derivatives of the time dependent log partition function $F(t)$, directly transfers to strong multiplicative comparison between the partition functions $Z_0$ and $Z_1$. 
In fact,  we have that $F'(0)=0$, while controlling the second order derivative $F''$ will require substantially more work. 
To see that $F'(0)=0$, we compute the derivative of $F(t)$ 
\[
F'(t) = \frac{Z_t'}{Z_t} = -\frac{
\E_{\gamma_n}\![
D\mathcal V(B_t)[Y]\,
e^{-\mathcal V(B_t)}
]
}{
\E_{\gamma_n}\![
e^{-\mathcal V(B_t)}
]
}
=
-\E_{\mu_t}\![D\mathcal V(B_t)[Y]],
\]
where $Y=B_{\mathrm o}$ and $d\mu_t(g) = Z_t^{-1}e^{-\mathcal V(B_t(g))}\,d\gamma_n(g)$.
Now, since $\mathcal V(X)=\Tr v(X)$ is a spectral trace function, we have that
$D\mathcal V(X)[Y]
=
\Tr\!(v'(X)Y)$.
At $t=0$, the matrix $B_0=B_{\mathrm d}$ is diagonal, and hence
$v'(B_{\mathrm d})$ is diagonal as well. On the other hand,
$B_{\mathrm o}$ has zero diagonal, and therefore
\[
D\mathcal V(B_{\mathrm d})[B_{\mathrm o}]
=
\Tr\!(v'(B_{\mathrm d})B_{\mathrm o})
=
0,
\]
which proves that $F'(0)=0$.

In the remainder, our main goal will be  then to establish the uniform lower bound
\[
F''(t)\ge -C\delta n,
\qquad
0\le t\le1,
\]
for some tunable parameter $\delta$, which then by Taylor's formula would give $F(1) - F(0) \ge -\frac{C\delta n}{2}$, or 
equivalently, $Z_1
\ge
e^{-C\delta n/2}Z_0.$
Consequently, a lower bound of the form
\[
Z_0\ge e^{-C_0 \delta n}
\]
for the diagonal partition function would immediately yield
\[
Z_1\ge e^{-O(\delta n)}
\]
for the full matrix $B$. Thus, the overall goal of the interpolation argument is to show that
introducing the off-diagonal entries decreases the partition function by
at most an exponential factor of order $e^{-O(n)}$.

\subsection{Construction of  Potential Function}
\label{sec:potential}

In this section, we will construct a potential function  of the form $\mathcal V(B)= \Tr v(B)$ for some suitable scalar function $v$, satisfying $0\le e^{-\mathcal V(B)}
\le
\one_{\{\opnorm B<1\}}$. Equivalently, $\mathcal V(B)$ should blow up when $B$ approaches the boundary of the operator-norm ball. Moreover, we will require that $\mathcal V(B)$ has enough convexity to control second order terms appearing in the interpolation argument. At the same time, since $Z_t = \E e^{-\mathcal V(B_t)}$ we cannot  make $\mathcal V(B)$ arbitrarily large: doing so  could make  the $Z_0$ partition function too small.

We now turn to the construction of the spectral potential. A natural
starting point is the scalar logarithmic barrier
\begin{equation*}
\ell(t)
=
-\log(1-t^2),
\qquad |t|<1,
\end{equation*}
which we extend by $+\infty$ outside $(-1,1)$. For a symmetric matrix
$B$, the corresponding spectral potential is
\[
\Tr \ell(B)
=
-\log\det(I-B^2),
\]
whenever $\opnorm{B}<1$. Since, $\ell(t)\longrightarrow+\infty$
as $|t|\uparrow1$, we have that  $e^{-\Tr\ell(B)}
=
\det(I-B^2)$ vanishes as the spectrum of $B$ approaches the boundary of the
operator-norm unit ball. Extending $\ell$ by $+\infty$ outside
$(-1,1)$ therefore ensures that $e^{-\Tr\ell(B)}$ is
supported on the event $\{\opnorm{B}<1\}$. Moreover, $\ell$ is convex, since 
$\ell''(t)
=
\frac{2(1+t^2)}{(1-t^2)^2}
>0$ for $|t|<1.$
Thus the second derivative of the logarithmic barrier grows rapidly as
$t$ approaches either boundary point $\pm1$. Consequently, the Hessian of
the associated matrix potential becomes large in spectral directions
corresponding to eigenvalues close to the boundary of the operator-norm
ball. This is precisely the type of curvature that will later be used
to control the variation of the partition function along the
 aforementioned interpolation path.
However, a naive use of the full logarithmic barrier can make the
partition function too small even for matrices whose eigenvalues lie well
inside $(-1,1)$, although the strong growth of the potential is only
needed near the spectral boundary. In particular, it may
prevent us from obtaining the required lower bound on the diagonal
partition function $Z_0$. This motivates separating the behavior of the logarithmic barrier in the spectral bulk from its behavior close to the boundary. To this end, using the power
series expansion
\begin{equation*}
\ell(t)
=
\nlsum_{k=1}^{\infty}\frac{t^{2k}}{k},
\qquad |t|<1,
\end{equation*}
we define, for an integer $q\ge2$, the degree-$\ge 2q$ tail
\begin{equation*}
r_q(t)
=
\nlsum_{k=q}^{\infty}\frac{t^{2k}}{k}
=
\ell(t)-\nlsum_{k=1}^{q-1}\frac{t^{2k}}{k}.
\end{equation*}
The tail $r_q$ retains the divergence of $\ell$ as $|t|\uparrow1$,
while for values of $|t|$ bounded away from $1$, the tail function $r_q(t)$
remains comparatively small because it only contains higher-order terms. Our idea is therefore to combine the tail $r_q$ with the complementary
low-order part of the logarithmic barrier, assigning different weights
to the two components. This allows us to strengthen the barrier near the
spectral boundary while keeping its effect in the bulk comparatively
small. More precisely, for parameters $\alpha\ge\beta>0$, we define
\begin{equation*}
v_{q,\alpha,\beta}(t)
=
\beta\,\ell(t)
+
(\alpha-\beta)r_q(t).
\end{equation*}
Equivalently,
\begin{equation*}
v_{q,\alpha,\beta}(t)
=
\beta\nlsum_{k<q}\frac{t^{2k}}{k}
+
\alpha\nlsum_{k\ge q}\frac{t^{2k}}{k}.
\end{equation*}
Thus the low-degree part of the logarithmic barrier is weighted only by
$\beta$, whereas the high-degree tail carries the larger coefficient
$\alpha$. The parameters $q$, $\alpha$, and $\beta$ will be chosen later
so that this potential function simultaneously provides the curvature required
for the interpolation argument and ensures that  the diagonal partition
function remains sufficiently large, namely $Z_0\ge e^{-O(n)}$.
Consequently,  we will set 
\begin{equation} \label{eq:potentialfunction}
   \mathcal V:\Sym_d(\R)\to(-\infty,+\infty],\qquad B \longmapsto \mathcal V(B)
=
\begin{cases}
\Tr v_{q,\alpha,\beta}(B), & \opnorm{B}<1,\\[1ex]
+\infty, & \opnorm{B}\ge 1,
\end{cases}
\end{equation}
where $v_{q,\alpha,\beta}(B)$ is defined by functional calculus, and we will assume in the following that $\alpha\ge2\beta$.
In particular, $0\le e^{-\mathcal V(B)}
\le
\one_{\{\opnorm{B}<1\}}$.
Moreover, since $v_{q,\alpha,\beta}$ is even and convex,
$\mathcal V$ is an even convex spectral function.

\subsection{Curvature Bounds}
\label{sec:analyticpotential}
In the following we will derive several useful analytic properties of the spectral potential function $\mathcal V(B)=\Tr v_{q,\alpha,\beta}(B)$ that we defined in \eqref{eq:potentialfunction}. We will begin with computing its
Hessian. To this end, in the lemma below, we record the general formula for Hessians of spectral trace functions.  

\begin{lemma}[Hessian of a spectral trace function]
\label{lem:spectral-hessian}
Let $I\subseteq\mathbb R$ be an open interval, let $f\in C^2(I)$,
and let $B\in\Sym_d(\R)$ have spectrum contained in $I$. Write
$B=U\diag(\lambda_1,\dots,\lambda_d)U^\top$ for the spectral decomposition of $B$.
Then, for every $Y\in\Sym_d(\R)$,
\begin{equation}
D^2\Tr f(B)[Y,Y]
=
\nlsum_{a,b=1}^d
f'^{[1]}(\lambda_a,\lambda_b)
|(U^\top YU)_{ab}|^2,
\label{eq:spectral-hessian}
\end{equation}
where
\[
f'^{[1]}(x,y)
=
\begin{cases}
\dfrac{f'(x)-f'(y)}{x-y},&x\neq y,\\[1ex]
f''(x),&x=y.
\end{cases}
\]
\end{lemma}

\begin{proof}
Since $D\Tr f(B)[Y]=\Tr\bigl(f'(B)Y\bigr),$
we have
\[
D^2\Tr f(B)[Y,Y]
=
\Tr\bigl(Df'(B)[Y]\,Y\bigr).
\]
Write  $B=U\diag(\lambda_1,\dots,\lambda_d)U^\top$ for the spectral decomposition of $B$. Now, by the Daleckii--Krein formula
\cite[Theorem V.3.3]{BhatiaMatrixAnalysis},
applied to the function $f'$, we obtain that
\[
Df'(B)[Y]
=
U\Bigl(
\bigl[f'^{[1]}(\lambda_a,\lambda_b)\bigr]_{a,b}
\circ(U^\top YU)
\Bigr)U^\top.
\]

Therefore, 
\[
D^2\Tr f(B)[Y,Y]
=
\nlsum_{a,b=1}^d
f'^{[1]}(\lambda_a,\lambda_b)
|(U^\top YU)_{ab}|^2,
\]

which yields \eqref{eq:spectral-hessian}.
\end{proof}
Since $v_{q,\alpha,\beta}'^{[1]}(x,y)
=
\beta\,\ell'^{[1]}(x,y)
+
(\alpha-\beta)\,r_q'^{[1]}(x,y)$, the Hessian of $\mathcal V$ is determined by the divided
differences of $\ell'$ and $r_q'$. Therefore, we  record in the following estimates on these two scalar functions, which we will need later in our arguments. We first compute $\ell'^{[1]}(x,y)$. Using that $\ell'(t)
=
\frac{1}{1-t}-\frac{1}{1+t}$,  we obtain

\begin{equation}
\ell'^{[1]}(x,y)
=
\frac{1}{(1-x)(1-y)}
+
\frac{1}{(1+x)(1+y)},
\qquad x,y\in(-1,1).
\label{eq:log-barrier-divided-difference}
\end{equation}

Moreover, we provide  estimates on
$r_q'^{[1]}$ in the following lemma below. 

\begin{lemma}[Divided-difference estimates]
\label{lem:tail-divided-difference}
There exist universal constants $a_0,c_0,C_0>0$ such that, for every
$q\ge2$ and every $x,y\in(-1,1)$,
\begin{align}
r_q'^{[1]}(x,y)
&\ge
c_0\Bigl(
\frac{
\one_{\{\min(1-x,\,1-y)\le a_0/q\}}
}{
(1-x)(1-y)
}
+
\frac{
\one_{\{\min(1+x,\,1+y)\le a_0/q\}}
}{
(1+x)(1+y)
}
\Bigr),
\label{eq:tail-divided-difference-lower}\\
r_q'^{[1]}(x,y)
&\le
C_0 q\,
\max\{|x|,|y|\}^{2q-2}
\ell'^{[1]}(x,y).
\label{eq:tail-divided-difference-upper}
\end{align}
Moreover, the universal constants can be chosen as $a_0= \frac{1}{64}$, $c_0=\frac{1}{4}$ and $C_0=2$.
\end{lemma}

Combining these estimates with Lemma~\ref{lem:spectral-hessian}, we can
now translate the scalar properties of $v_{q,\alpha,\beta}$ into
quantitative lower bounds for the Hessian of $\mathcal V$. To express
these bounds conveniently, in the following we introduce  some notation below.

For $\rho\in\{\pm1\}$, let 
\begin{equation}
    \label{eq:def-hrho}
    H_\rho(B):=I+\rho B
\end{equation}
 Since $\opnorm{B}<1$, both $H_+$ and $H_-$ are positive definite.
Moreover, for $Y\in\Sym_d(\R)$,  we define 
\begin{equation}
    \label{eq:def-erho}
    E^\rho(Y) := H_\rho^{-1/2}(\rho Y)H_\rho^{-1/2}
\end{equation}
Now, let $B=U\diag(\lambda_1,\dots,\lambda_d)U^\top$ be the spectral decomposition of $B$ and write
$\widetilde Y=U^\top YU$, then using \eqref{eq:log-barrier-divided-difference} and
Lemma~\ref{lem:spectral-hessian}, we obtain 
\begin{equation}
D^2\Tr\ell(B)[Y,Y]
=
\nlsum_{\rho\in\{\pm1\}}
\Tr\bigl(E^\rho(Y)^2\bigr) = \nlsum_{\rho\in\{\pm1\}}\nlsum_{a,b=1}^d
\frac{|\widetilde Y_{ab}|^2}
{(1+\rho\lambda_a)(1+\rho\lambda_b)}.
\label{eq:log-barrier-hessian}
\end{equation}

Let $P_\rho = \one_{\{H_\rho\le a_0q^{-1}I\}}$ denote the projection onto the eigenspace of $B$ whose eigenvalues lie within distance $a_0/q$ of $-\rho$, and let $Q_\rho=I-P_\rho$.  Furthermore, define
\begin{equation}
\label{eq:bulk-normalized-direction}
    E_{\mathrm{bulk}}^\rho(Y)
:=
Q_\rho E^\rho(Y)Q_\rho,
\end{equation}
and
\begin{equation}
  E_{\mathrm{bdry}}^\rho(Y)
:=
E^\rho(Y)-E_{\mathrm{bulk}}^\rho(Y).
\label{eq:boundary-normalized-direction}  
\end{equation}

Now, using Lemma~\ref{lem:tail-divided-difference} and
Lemma~\ref{lem:spectral-hessian} we get that,

\begin{align*}
D^2\Tr r_q(B)[Y,Y]
&\ge
c_0\nlsum_{\rho\in\{\pm1\}}\nlsum_{a,b=1}^d
\frac{
\one_{\{
\min(1+\rho\lambda_a,\,1+\rho\lambda_b)\le a_0/q
\}}
}{
(1+\rho\lambda_a)(1+\rho\lambda_b)
}
|\widetilde Y_{ab}|^2 \\
&= c_0\nlsum_{\rho\in\{\pm1\}}\nlsum_{a,b=1}^d\one_{\{
\min(1+\rho\lambda_a,\,1+\rho\lambda_b)\le a_0/q 
\}} |E^\rho(Y)_{ab}|^2,
\end{align*}

where used that $|E^\rho(Y)_{ab}|^2
=
\frac{|\widetilde Y_{ab}|^2}
{(1+\rho\lambda_a)(1+\rho\lambda_b)}.$
Consequently, as the matrix
$E_{\mathrm{bdry}}^\rho(Y)$ consists, in an eigenbasis of $B$, precisely of the entries corresponding to pairs $(a,b)$ for which $\min\{1+\rho\lambda_a,1+\rho\lambda_b\}
\le \frac{a_0}{q}$, we then obtain
\begin{equation}
D^2\Tr r_q(B)[Y,Y]
\ge
c_0
\nlsum_{\rho\in\{\pm1\}}
\norm{E_{\mathrm{bdry}}^\rho(Y)}_F^2.
\label{eq:tail-hessian-lower}
\end{equation}

\begin{lemma}[Lower bound for the Hessian of the Potential]
\label{lem:potential-hessian-lower} 
Let $\mathcal V(B)=\Tr v_{q,\alpha,\beta}(B)$ be defined as in \eqref{eq:potentialfunction}, and let $Y,B\in\Sym_d(\R)$ with
$\opnorm{B}<1$. 
 Then, we have
\begin{equation}
D^2\mathcal V(B)[Y,Y]
\ge
\nlsum_{\rho\in\{\pm1\}}
\bigl(
c_0\alpha
\norm{E_{\mathrm{bdry}}^\rho(Y)}_F^2
+
\beta
\norm{E_{\mathrm{bulk}}^\rho(Y)}_F^2
\bigr),
\label{eq:potential-hessian-lower}
\end{equation}
where $c_0$ is the same universal constant as in Lemma \ref{lem:tail-divided-difference}.
\end{lemma}

Now, let $M_1,\dots,M_n$ be a family of symmetric matrices. It will be convenient to introduce the following notation. For $\rho\in\{\pm1\}$, we write
\[
E_i^\rho := E^\rho(M_i)
=
H_\rho^{-1/2}(\rho M_i)H_\rho^{-1/2},
\]
where $H_\rho$ and $E^\rho$ are defined as in \eqref{eq:def-hrho} and \eqref{eq:def-erho} respectively. 

Moreover, write 
\begin{equation}
    \label{eq:def-ebulk}
  E_{i,\mathrm{bulk}}^\rho
:=
E_{\mathrm{bulk}}^\rho(M_i)
=
Q_\rho E_i^\rho Q_\rho,
\qquad
E_{i,\mathrm{bdry}}^\rho
:=
E_{\mathrm{bdry}}^\rho(M_i)
=
E_i^\rho-E_{i,\mathrm{bulk}}^\rho,
\end{equation}
and  for $z = (z_1,\dots,z_n)$ with $z_i\in \R^r$ define the Gram forms,

\begin{align}
\mathfrak h_\rho(z)
&:=
\nlsum_{i,j=1}^n
\Tr\!(
E_{i,\mathrm{bdry}}^\rho
E_{j,\mathrm{bdry}}^\rho
)
\langle z_i,z_j\rangle,
\label{eq:hard-edge-gram-form}\\
\mathfrak u_\rho(z)
&:=
\nlsum_{i,j=1}^n
\Tr\!(
E_{i,\mathrm{bulk}}^\rho
E_{j,\mathrm{bulk}}^\rho
)
\langle z_i,z_j\rangle.
\label{eq:bulk-gram-form}
\end{align}

We will need later the following simple consequence of Lemma~\ref{lem:potential-hessian-lower}.

\begin{lemma}
\label{lem:barrier-hessian-coercivity}
Let $\mathcal V(B)=\Tr v_{q,\alpha,\beta}(B)$ be defined as in \eqref{eq:potentialfunction}, and let $M_1,\dots,M_n$ be symmetric matrices. Moreover, let $B\in\Sym_d(\R)$  be a symmetric matrix satisfying
$\opnorm{B}<1$. Then, for every finite-dimensional Hilbert space $\mathcal H$
and every family $z_1,\dots,z_n\in\mathcal H$,
\begin{equation}
\nlsum_{i,j=1}^n
D^2\mathcal V(B)[M_i,M_j]\,
\langle z_i,z_j\rangle
\ge
c_0\alpha
\nlsum_{\rho\in\{\pm1\}}\mathfrak h_\rho(z)
+
\beta
\nlsum_{\rho\in\{\pm1\}}\mathfrak u_\rho(z),
\label{eq:barrier-hessian-coercivity}
\end{equation}
where $c_0$ is the same universal constant as in Lemma \ref{lem:potential-hessian-lower}.
\end{lemma}

Next, we record a useful simple technical lemma. 

\begin{lemma}
\label{lem:cyclic-tensor-bounds}
Let $C_1,\dots,C_n,D_1,\dots,D_n\in\Sym_d(\R)$ and let $z_1,\dots,z_n\in\R^d\otimes\mathcal H$, where
$\mathcal H$ is  some finite-dimensional Hilbert space. Then, we have that
\begin{equation}
\Bigl|
\nlsum_{i,j=1}^n
\langle
z_i,\,
(C_jC_i\otimes I_{\mathcal H})z_j
\rangle
\Bigr|
\le
\nlsum_{i,j=1}^n
\Tr(C_iC_j)\,
\langle z_i,z_j\rangle,
\label{eq:cyclic-tensor-pure}
\end{equation}
and
\begin{equation}
    \Bigl|
\nlsum_{i,j=1}^n
\langle
z_i,\,
(C_jD_i\otimes I_{\mathcal H})z_j
\rangle
\Bigr|
\le
\Bigl(
\nlsum_{i,j=1}^n
\Tr(C_iC_j)\,
\langle z_i,z_j\rangle
\Bigr)^{1/2}
\Bigl(
\nlsum_{i,j=1}^n
\Tr(D_iD_j)\,
\langle z_i,z_j\rangle
\Bigr)^{1/2}.
\label{eq:cyclic-tensor-mixed}
\end{equation}
\end{lemma}

In the following technical lemma, we establish a variance bound on the matrix family $E_{1,\mathrm{bulk}}^\rho, \dots,E_{n,\mathrm{bulk}}^\rho$.

\begin{lemma}[Variance bound]
\label{lem:pure-bulk-variance}
Let $M_1,\dots,M_n$ be some family of symmetric matrices satisfying $\nlsum_{i=1}^nM_i^2
\preceq
\eta I_d$, and let $B\in\Sym_d(\R)$  be a symmetric matrix satisfying
$\opnorm{B}<1$.
 Then, there is some universal constant $C>0$ such that 
\begin{equation}
\nlsum_{i=1}^n
(E_{i,\mathrm{bulk}}^\rho)^2
\preceq
Cq^2\eta I_d,
\label{eq:pure-bulk-variance}
\end{equation}
where $E_{1,\mathrm{bulk}}^\rho, \dots,E_{n,\mathrm{bulk}}^\rho$ are defined as in \eqref{eq:def-ebulk}, and where one can take $C=\frac{1}{a_0^2}=4096$.
Moreover, for every finite-dimensional Hilbert space $\mathcal H$
and every family $z_1,\dots,z_n\in\R^d\otimes\mathcal H$ we moreover have,
\begin{equation}
\Bigl|
\nlsum_{i,j=1}^n
\bigl\langle
z_i,\,
(
E_{j,\mathrm{bulk}}^\rho
E_{i,\mathrm{bulk}}^\rho
\otimes I_{\mathcal H}
)z_j
\bigr\rangle
\Bigr|
\le
Cq^2\eta
\nlsum_{i=1}^n\norm{z_i}_2^2.
\label{eq:pure-bulk-cyclic-bound}
\end{equation}
\end{lemma}

Let  $\sigma,\tau\in\{\pm1\}$, and denote

\begin{equation}
\mathcal A_{\sigma\tau}
:=
H_\tau\otimes H_\sigma,
\qquad
\mathcal K_{\sigma\tau}
:=
\mathcal A_{\sigma\tau}^{-1}
=
H_\tau^{-1}\otimes H_\sigma^{-1}
\label{eq:two-slot-inverse-weight}
\end{equation}
where $H_\sigma$ is defined as in \eqref{eq:def-hrho}. Now, let $B= B(g)= \nlsum_{i=1}^ng_iM_i$.
On the event $\Omega=\{g\in\R^n:\opnorm{B(g)}<1\}$ we 
let
\begin{equation}
W_{\sigma\tau}(g)
:=
e^{-\norm{g}^2/2-\mathcal V(B(g))}
\mathcal K_{\sigma\tau}(g).
\label{eq:two-slot-weight}
\end{equation}

Moreover, for $\rho\in\{\pm1\}$ we denote  
\begin{equation}
\label{eq:one-slot-weight}
 W_\rho(g)
=
e^{-\norm{g}^2/2-\mathcal V(B(g))}H_\rho^{-1}.
\end{equation}

Recall from \eqref{eq:matrix-curvature-definition} that the curvature block matrix of a
positive matrix weight $W$ is given by 
\[
[\Theta_{ij}(W)]_{ij}, \qquad\Theta_{ij}(W)
=
(\partial_iW)W^{-1}(\partial_jW)-\partial_{ij}W.
\]
It will be convenient to work with the corresponding normalized
curvature blocks
\begin{equation}
\widetilde\Theta_{ij}(W)
:=
W^{-1/2}
\Theta_{ij}(W)
W^{-1/2}.
\label{eq:normalized-two-slot-curvature-definition}
\end{equation}

In the following our goal will be to show that
\[
[\widetilde\Theta_{ij}^{\sigma\tau}]_{i,j=1}^n
\succeq \frac{1}{2} I, \qquad \text{and} \qquad [\widetilde\Theta_{ij}^{\rho}]_{i,j=1}^n
\succeq \frac{1}{2} I
\]
where $[\widetilde\Theta_{ij}^{\sigma\tau}]_{ij}:=[\widetilde\Theta_{ij}(W_{\sigma\tau})]_{ij}$ and $[\widetilde\Theta_{ij}^{\rho}]_{ij}:=[\widetilde\Theta_{ij}(W_{\rho})]_{ij}$ denote the normalized curvature block matrix associated with $W_{\sigma\tau}$, respectively, $W_{\rho}$.

\begin{proposition}[Curvature lower bound]
\label{prop:two-slot-curvature} 
  There exist universal constants $A_*,D_*,c_{\mathrm{curv}}>0$ such that the following holds. 
Suppose that 
\begin{equation}
\alpha\ge\beta, \qquad
\alpha\ge A_*,
\qquad
\alpha\beta\ge D_*,
\qquad
q^2\eta\le c_{\mathrm{curv}},
\label{eq:curvature-parameter-conditions}
\end{equation}
Let $\mathcal V(B)=\Tr v_{q,\alpha,\beta}(B)$ be defined as in \eqref{eq:potentialfunction}, and let $B(g)= \nlsum_{i=1}^ng_iM_i$ with symmetric coefficient matrices $M_i$ satisfying $\nlsum_{i=1}^nM_i^2
\preceq
\eta I_d$. Let $W_{\sigma\tau}$ and $W_{\rho}$ be defined as in \eqref{eq:two-slot-weight}, respectively, \eqref{eq:one-slot-weight}, and let 
  $[\widetilde\Theta_{ij}^{\sigma\tau}]_{i,j=1}^n$ and $[\widetilde\Theta_{ij}^{\rho}]_{i,j=1}^n$ denote the normalized curvature block matrix associated with $W_{\sigma\tau}$, respectively, $W_{\rho}$. 
  
Then, for any
$\rho,\sigma,\tau\in\{\pm1\}$ we have that, 
\begin{equation}
[\widetilde\Theta_{ij}^{\sigma\tau}]_{i,j=1}^n
\succeq \frac{1}{2} I, \qquad \text{and} \qquad [\widetilde\Theta_{ij}^{\rho}]_{i,j=1}^n
\succeq \frac{1}{2} I
\label{eq:two-slot-curvature-gap}
\end{equation}
Moreover, the absolute constants can be chosen as $A_*=32$, $D_*=128$ and $c_{\mathrm{curv}}= \frac{1}{16384}$.
\end{proposition}

\begin{proof}
We first compute the normalized curvature of $W_{\sigma\tau}$.
Using that $\partial_iB=M_i$ and 
\eqref{eq:scalar-curvature-transformation}, 
 differentiation of
$\mathcal K_{\sigma\tau}
=
H_\tau^{-1}\otimes H_\sigma^{-1}$
gives
\begin{align}
\widetilde\Theta_{ij}^{\sigma\tau}
={}&
\delta_{ij}I_{d^2}
+
D^2\mathcal V(B)[M_i,M_j]I_{d^2}
\nonumber\\
&-
E_j^\tau E_i^\tau\otimes I_d
-
I_d\otimes E_j^\sigma E_i^\sigma.
\label{eq:normalized-two-slot-curvature}
\end{align}
Let
$z_1,\dots,z_n\in\R^d\otimes\R^d$.  Using $E_i^\rho
=
E_{i,\mathrm{bdry}}^\rho
+
E_{i,\mathrm{bulk}}^\rho,$
we expand
\begin{align*}
   \nlsum_{i,j}
\bigl\langle
z_i,
(E_j^\rho E_i^\rho\otimes I_d)z_j
\bigr\rangle
=&
\nlsum_{i,j}
\bigl\langle
z_i,
(
E_{j,\mathrm{bdry}}^\rho
E_{i,\mathrm{bdry}}^\rho
\otimes I_d
)z_j
\bigr\rangle
+
\nlsum_{i,j}
\bigl\langle
z_i,
(
E_{j,\mathrm{bdry}}^\rho
E_{i,\mathrm{bulk}}^\rho
\otimes I_d
)z_j
\bigr\rangle \\
&+
\nlsum_{i,j}
\bigl\langle
z_i,
(
E_{j,\mathrm{bulk}}^\rho
E_{i,\mathrm{bdry}}^\rho
\otimes I_d
)z_j
\bigr\rangle
+
\nlsum_{i,j}
\bigl\langle
z_i,
(
E_{j,\mathrm{bulk}}^\rho
E_{i,\mathrm{bulk}}^\rho
\otimes I_d
)z_j
\bigr\rangle.
\end{align*}
Therefore, using  Lemma~\ref{lem:cyclic-tensor-bounds} and
Lemma~\ref{lem:pure-bulk-variance} we obtain,
\begin{equation}
\Bigl|
\nlsum_{i,j}
\bigl\langle
z_i,
(E_j^\rho E_i^\rho\otimes I_d)z_j
\bigr\rangle
\Bigr|
\le
\mathfrak h_\rho(z)
+
2\sqrt{\mathfrak h_\rho(z)\mathfrak u_\rho(z)}
+
C_{\ref{lem:pure-bulk-variance}}q^2\eta\nlsum_i\norm{z_i}^2.
\label{eq:one-slot-cyclic-bound}
\end{equation}

Recall that $c_0=\frac14$ is the constant appearing in Lemma~\ref{lem:barrier-hessian-coercivity}.  Young's inequality gives
\begin{equation}
2\sqrt{\mathfrak h_\rho\mathfrak u_\rho}
\le \frac{c_0\alpha}{8} \,\mathfrak h_\rho
+ \frac{8}{c_0\alpha} \mathfrak u_\rho,
\label{eq:hard-bulk-young}
\end{equation}
and therefore $\mathfrak h_\rho + 2\sqrt{\mathfrak h_\rho\mathfrak u_\rho} \le (1+\frac{c_0\alpha}{8} )\,\mathfrak h_\rho + \frac{8}{c_0\alpha} \mathfrak u_\rho.$ If $\alpha\ge\frac{8}{c_0}=32$,  then $1\le\frac{c_0\alpha}{8}$, so that $1+\frac{c_0\alpha}{8}\le \frac{c_0\alpha}{4}$.
Likewise if $\alpha\beta\ge \frac{32}{c_0}=128$, then $\frac{8}{c_0\alpha}\le\frac{\beta}{4}$.
Now, under this choice of $\alpha$ and $\beta$ then \eqref{eq:one-slot-cyclic-bound} simplifies to
\begin{equation}
\Bigl|
\nlsum_{i,j}
\bigl\langle
z_i,
(E_j^\rho E_i^\rho\otimes I_d)z_j
\bigr\rangle
\Bigr|
\le
\frac{c_0\alpha}{4}\,\mathfrak h_\rho(z)
+
\frac{\beta}{4}\mathfrak u_\rho(z)
+
C_{\ref{lem:pure-bulk-variance}}q^2\eta\nlsum_i\norm{z_i}^2.
\label{eq:one-slot-final-absorption}
\end{equation}

Now, using
Lemma~\ref{lem:barrier-hessian-coercivity} gives
\begin{equation}
    \nlsum_{i,j}
\bigl\langle
z_i,
(
E_j^\tau E_i^\tau\otimes I_d
+
I_d\otimes E_j^\sigma E_i^\sigma
)z_j
\bigr\rangle
\le
\frac12
\nlsum_{i,j}
D^2\mathcal V(B)[M_i,M_j]
\langle z_i,z_j\rangle
+
2C_{\ref{lem:pure-bulk-variance}}q^2\eta\nlsum_i\norm{z_i}^2.
\label{eq:negative-curvature-absorption}
\end{equation}
Therefore, by substituting \eqref{eq:negative-curvature-absorption} into
\eqref{eq:normalized-two-slot-curvature}, we obtain
\begin{align*}
\nlsum_{i,j}
\langle
z_i,\widetilde\Theta_{ij}^{\sigma\tau}z_j
\rangle
&\ge
\nlsum_i\norm{z_i}^2
+
\frac12
\nlsum_{i,j}
D^2\mathcal V(B)[M_i,M_j]
\langle z_i,z_j\rangle
\\
&\qquad
-
2C_{\ref{lem:pure-bulk-variance}}q^2\eta\nlsum_i\norm{z_i}^2.
\end{align*}
Since $\mathcal V$ is convex, its  Hessian is
nonnegative, and hence $\nlsum_{i,j}
\langle
z_i,\widetilde\Theta_{ij}^{\sigma\tau}z_j
\rangle
\ge
(1-2C_{\ref{lem:pure-bulk-variance}}q^2\eta)
\nlsum_i\norm{z_i}^2.$
Finally, choosing $q^2\eta\le\frac{1}{4C_{\ref{lem:pure-bulk-variance}}}$   ensures that

\[
\nlsum_{i,j}
\langle
z_i,\widetilde\Theta_{ij}^{\sigma\tau}z_j
\rangle
\ge
\frac12\nlsum_i\norm{z_i}^2,
\]

which proves \eqref{eq:two-slot-curvature-gap}. The proof for $W_\rho$ follows by the same argument, which we omit.
\end{proof}

The curvature estimate from Proposition~\ref{prop:two-slot-curvature}
now allows us to apply the matrix-weighted Poincar\'e inequality with respect to the matrix weight $W_{\sigma\tau}$. More specifically, we consider smooth matrix-valued field
$F:\Omega\to\R^{d\times d}$ together with its vectorization $f(g)=\vecop(F(g))\in\R^{d^2}$ on the domain
$\Omega=\{g:\opnorm{B(g)}<1\}$, and our goal  will be to suitably upper bound the following quantity via a matrix weighted Poincar\'e inequality of Section \ref{sec:matrix-poincare},

\begin{equation}
\mathcal Q_{\sigma\tau}(F)
:=
\E_\mu
[
f^\top\mathcal K_{\sigma\tau}f
],
\label{eq:Q-reciprocal}
\end{equation}
where 
\begin{equation}
d\mu(g)
=
\frac{e^{-\mathcal V(B(g))}}
{\E_{\gamma_n}e^{-\mathcal V(B(g))}}
\,d\gamma_n(g).
\label{eq:reciprocal-tilted-measure}
\end{equation}

\begin{proposition}
\label{thm:defective-two-slot}
There exist universal constants $C>0$ and $c_{\mathrm P}\in(0,c_{\mathrm{curv}}]$, where $c_{\mathrm{curv}}$ is as in Proposition \ref{prop:two-slot-curvature}, such that the following holds. 
Suppose that $\alpha\ge\beta$,
$\alpha\ge A_*$,
$\alpha\beta\ge D_*$,
 and $q^2\eta\le c_{\mathrm P}$, where $A_*$ and $D_*$ are as in Proposition \ref{prop:two-slot-curvature}.
Let $\mathcal V(B)=\Tr v_{q,\alpha,\beta}(B)$ be defined as in \eqref{eq:potentialfunction}, and let $B(g)= \nlsum_{i=1}^ng_iM_i$ with symmetric coefficient matrices $M_i$ that satisfy $\nlsum_{i=1}^nM_i^2
\preceq
\eta I_d$. Then,
 for every matrix-valued polynomial $F:\R^n\to\R^{d\times d}$ (all quantities below being finite),
\begin{equation}
\mathcal Q_{\sigma\tau}(F)
\le
C\E_\mu\hsnorm{F}^2
+
C\nlsum_{i=1}^n
\mathcal Q_{\sigma\tau}(\partial_iF),
\label{eq:defective-two-slot}
\end{equation}
where $\mathcal Q_{\sigma\tau}(F)$ is defined as in \eqref{eq:Q-reciprocal}.
\end{proposition}

In the following our goal will be to apply Proposition~\ref{thm:defective-two-slot} iteratively to the matrix valued polynomial 
\[
P(g)=B(g)^{q-1}Y(g),
\]
where $B(g)=\nlsum_{i=1}^n g_iM_i$.

To control Gaussian matrix trace moments that arise in the context of Proposition~\ref{thm:defective-two-slot}, we record the following lemma; see, for example \cite{Tropp2012}.

\begin{lemma}[Gaussian matrix trace moment]
\label{lem:gaussian-trace-moment}
Let $Z=\nlsum_i h_iM_i$
be a self-adjoint $d\times d$ Gaussian matrix series, where the $h_i$
are independent standard Gaussian random variables, and suppose that $\nlsum_iM_i^2\preceq\nu I_d.$
Then, for every integer $p\ge1$,
\begin{equation}
\E\Tr|Z|^{2p}
\le
d(Cp\nu)^p,
\label{eq:gaussian-trace-moment}
\end{equation}
where $C>0$ is a universal constant.
\end{lemma}

We next estimate the moments of derivatives of $P=B^{q-1}Y$.  
For a matrix-valued field $P$, we use the notation
\begin{equation}
\norm{\nabla^jP(g)}_{\mathrm{HS}}^2
:=
\nlsum_{i_1,\dots,i_j=1}^n
\hsnorm{
\partial_{i_1}\cdots\partial_{i_j}P(g)
}^2,
\label{eq:derivative-HS-norm}
\end{equation}
with the convention
$\norm{\nabla^0P}_{\mathrm{HS}}=\hsnorm{P}$.

\begin{lemma}
\label{lem:polynomial-derivatives}
Let $B(g)=\nlsum_{i=1}^n g_iM_i$ and let $Y(g)=\nlsum_{i=1}^n g_iN_i$ have symmetric 
 coefficient matrices  and that satisfy $\nlsum_{i=1}^nM_i^2\preceq\eta I_d$ and $\nlsum_{i=1}^nN_i^2\preceq C_0\eta I_d$
for some universal constant $C_0>0$.  Let $P(g)=B(g)^{q-1}Y(g).$
Then, for every $0\le j\le q$,
\begin{equation}
\E_\mu\norm{\nabla^jP}_{\mathrm{HS}}^2
\le
q^{2j}d(Cq\eta)^q
\le
d(Cq^3\eta)^q,
\label{eq:polynomial-derivative-bound}
\end{equation}
where $C$ is a universal constant that depends only on $C_0$.
\end{lemma}

 We now apply Proposition~\ref{thm:defective-two-slot} iteratively to the matrix valued polynomial 
$P(g)=B(g)^{q-1}Y(g),$ where $B(g)=\nlsum_{i=1}^n g_iM_i$ and $Y(g)=\nlsum_{i=1}^n g_iN_i$.

\begin{lemma}
\label{lem:polynomial-two-slot}
For every fixed $C_0>0$, there exists a constant
$C=C(C_0)>0$ such that the following holds.
Suppose that $\alpha\ge\beta$,
$\alpha\ge A_*$,
$\alpha\beta\ge D_*$,
$q^2\eta\le c_{\mathrm{P}}$, where $A_*$ and $D_*$  are as in Proposition \ref{prop:two-slot-curvature} and $c_P$ is defined as in Proposition \ref{thm:defective-two-slot}.
Let $B(g)=\nlsum_{i=1}^n g_iM_i$ and let $Y(g)=\nlsum_{i=1}^n g_iN_i$ have symmetric 
 coefficient matrices  and that satisfy $\nlsum_{i=1}^nM_i^2\preceq\eta I_d$ and $\nlsum_{i=1}^nN_i^2\preceq C_0\eta I_d$.   Then, we have
\begin{equation}
\mathcal Q_{\sigma\tau}(B^{q-1}Y)
\le
d(Cq^3\eta)^q, \qquad \mathcal Q_{\sigma\tau}(Y)
\le
Cd\eta,
\label{eq:polynomial-two-slot}
\end{equation}
where $\mathcal Q_{\sigma\tau}$ is defined as in \eqref{eq:Q-reciprocal}.
\end{lemma}

\begin{proof}Let $P(g)=B(g)^{q-1}Y(g).$
Applying Proposition~\ref{thm:defective-two-slot} repeatedly to  a matrix-valued
polynomial $P$ of degree $r$, gives 
\begin{equation}
\mathcal Q_{\sigma\tau}(P)
\le
\nlsum_{j=0}^{r}
C^{j+1}
\E_\mu\norm{\nabla^jP}_{\mathrm{HS}}^2.
\label{eq:defective-iteration}
\end{equation}
Specifically, for $P=B^{q-1}Y$, using Lemma~\ref{lem:polynomial-derivatives}, we obtain from \eqref{eq:defective-iteration} that, 
\begin{align*}
\mathcal Q_{\sigma\tau}(B^{q-1}Y)
&\le
\nlsum_{j=0}^{q}
C^{j+1}
\E_\mu\norm{\nabla^jP}_{\mathrm{HS}}^2\\
&\le
\nlsum_{j=0}^{q}
C^{j+1}
d(Cq^3\eta)^q\le d(C^\prime q^3\eta)^q,
\end{align*}
where $C^\prime$ is some universal constant. 

Moreover, using again Proposition~\ref{thm:defective-two-slot} gives
\begin{equation}
\mathcal Q_{\sigma\tau}(Y)
\le
C\E_\mu\hsnorm{Y}^2
+
C\nlsum_{i=1}^n
\mathcal Q_{\sigma\tau}(N_i).
\label{eq:linear-two-slot-first}
\end{equation}

Now, using $\nlsum_{i=1}^nN_i^2\preceq C_0\eta I_d$ and Harg\'e's convex domination inequality \cite{Harge2004} we obtain,
\[
\E_\mu\hsnorm{Y}^2
\le
\E_{\gamma_n}\hsnorm{Y}^2
=
\nlsum_{i=1}^n\hsnorm{N_i}^2 
\le
C_0d\eta.
\]

Finally, applying Proposition~\ref{thm:defective-two-slot} once more  to $F=N_i$, gives $\mathcal Q_{\sigma\tau}(N_i)
\le
C\hsnorm{N_i}^2$, and therefore we have,
\[
\mathcal Q_{\sigma\tau}(Y)
\le
C^\prime d\eta
+
C\nlsum_i\hsnorm{N_i}^2 \le C^\prime d\eta,
\]
which concludes our proof.
\end{proof}

\subsection{Partition-Function Bounds via Interpolation}
\label{sec:interpolation}

We now return to the interpolation argument that we outlined at the beginning of Section \ref{sec:hereditary-small-ball}. Recall that $B(g) = \nlsum_{i=1}^n g_i M_i$ is the Gaussian matrix series for which we want to establish the matrix small ball estimate, and that the interpolation is given by $B_t(g) = B_d(g) + tB_{\mathrm{o}}(g)$, where $B_d$ denotes the diagonal part of $B$ while $B_{\mathrm o}= B-B_d$ refers to the off-diagonal part. Similarly, we also consider the diagonal and off-diagonal decomposition at the level of the coefficient matrices $M_i = M_{i,d} + M_{i,o}$, and we denote with $M_i(t)=M_{i,d} + tM_{i,o} $. Moreover, we will also use the notation $Y(g)=B_{\mathrm o}(g) = \nlsum_{i=1}^n g_iM_{i,\mathrm o} $. We denote the  partition function along the interpolation by $Z_t=\E e^{-\mathcal V(B_t)}$, and let $F(t)=\log Z_t$. In the lemma below we establish  regularity properties of $Z_t$ and $F(t)$. 

\begin{lemma}[Regularity of the partition function]
\label{lem:regularity}
Assume $\alpha>4$. Let
$\mathcal V$ be the potential defined in \eqref{eq:potentialfunction}. Then $Z,F\in C^2([0,1])$, and
\begin{equation}
F'(t)
=
-\E_{\mu_t}D\mathcal V(B_t)[Y],
\label{eq:log-partition-first-derivative}
\end{equation}
\begin{equation}
F''(t)
=
\Var_{\mu_t}
(
D\mathcal V(B_t)[Y]
)
-
\E_{\mu_t}
D^2\mathcal V(B_t)[Y,Y],
\label{eq:F-second-derivative}
\end{equation}
where $d\mu_t(g)
=
\frac{e^{-\mathcal V(B_t(g))}}{Z_t}\,d\gamma_n(g)$.
\end{lemma}

Recall, the goal of the interpolation argument is to establish $F''(t)\ge -O(\delta n)$. However, since $\Var_{\mu_t}\!(D\mathcal V(B_t)[Y])\ge0$ it suffices to prove $\E_{\mu_t}D^2\mathcal V(B_t)[Y,Y]\le O(\delta n)$.

\begin{lemma}[Curvature bound along interpolation path]
\label{lem:expected-interpolation-hessian}
There exists a universal constant $C>0$ such that the following holds.
Suppose that $\alpha\ge\beta$,
$\alpha\ge A_*$,
$\alpha\beta\ge D_*$,
$q^2\eta\le c_{\mathrm{P}}$, where $A_*$ and $D_*$  are as in Proposition \ref{prop:two-slot-curvature} and  $c_P$ is defined as in Proposition \ref{thm:defective-two-slot}.
Let $\mathcal V(B)=\Tr v_{q,\alpha,\beta}(B)$ be defined as in \eqref{eq:potentialfunction}, and
let $B(g)=\nlsum_{i=1}^n g_iM_i$  have symmetric coefficient matrices  that satisfy $\nlsum_{i=1}^nM_i^2\preceq\eta I_d$. Moreover, let $Y(g)=\nlsum_{i=1}^n g_iM_{i,\mathrm o}$, and denote  $B_t(g) = B_d(g) + tB_{\mathrm{o}}(g)$. Then, we have uniformly for $t\in[0,1]$,
\begin{equation}
\E_{\mu_t}D^2\mathcal V(B_t)[Y,Y]
\le
C\beta d\eta
+
C\alpha qd(Cq^3\eta)^q,
\label{eq:expected-interpolation-hessian}
\end{equation}
where   $d\mu_t(g)$ denotes the tilted Gaussian measure $d\mu_t(g)=\frac{e^{-\mathcal V(B_t(g))}}{Z_t}\,d\gamma_n(g)$.
\end{lemma}
\begin{proof}
Since $\mathcal V
=\beta\Tr\ell+(\alpha-\beta)\Tr r_q$ we estimate separately the contributions of the  logarithmic term $\ell(x)=-\log(1-x^2)$ and of the tail term $r_q$. Let $H_\rho=I_d+\rho B_t.$ Let $B_t = U\diag(\lambda_1,\dots,\lambda_d) U^\top$ and $\widetilde{Y} = U^\top YU$. Using \eqref{eq:log-barrier-divided-difference} we obtain,

\[
\begin{aligned}
D^2\Tr\ell(B_t)[Y,Y]
&=
\nlsum_{a,b=1}^d
\ell'^{[1]}(\lambda_a,\lambda_b)
|\widetilde Y_{ab}|^2\\
&=
\nlsum_{\rho=\pm1}
\nlsum_{a,b=1}^d
\frac{
|\widetilde Y_{ab}|^2
}{
(1+\rho\lambda_a)(1+\rho\lambda_b)
}\\
&=
\nlsum_{\rho=\pm1}
\vecop(Y)^\top
(
H_\rho^{-1}\otimes H_\rho^{-1}
)
\vecop(Y),
\end{aligned}
\]
and therefore 
\begin{equation}
\label{eq:hessianlog}
    \E_{\mu_t}D^2\Tr\ell(B_t)[Y,Y]
=
\nlsum_{\rho=\pm1}
\mathcal Q_{\rho\rho}^{(t)}(Y).
\end{equation}

Moreover, using Lemma \ref{lem:tail-divided-difference}, we obtain for Hessian of $r_q$,
\begin{align}
D^2\Tr r_q(B_t)[Y,Y]
&=
\nlsum_{a,b=1}^d
r_q'^{[1]}(\lambda_a,\lambda_b)
|\widetilde Y_{ab}|^2
\nonumber\\
&\le
Cq\nlsum_{\rho=\pm1}\nlsum_{a,b=1}^d
\frac{
\max\{|\lambda_a|,|\lambda_b|\}^{2q-2}
}{
(1+\rho\lambda_a)(1+\rho\lambda_b)
}
|\widetilde Y_{ab}|^2.
\label{eq:tail-hessian-spectral-bound}
\end{align}
Using $\max\{|\lambda_a|,|\lambda_b|\}^{2q-2}
\le
|\lambda_a|^{2q-2}+|\lambda_b|^{2q-2}$ and $|\widetilde Y_{ab}|^2
=|\widetilde Y_{ba}|^2$ we obtain,
\begin{align}
D^2\Tr r_q(B_t)[Y,Y]
&\le
Cq\nlsum_{\rho=\pm1}\nlsum_{a,b=1}^d
\frac{
\lambda_a^{2q-2}
}{
(1+\rho\lambda_a)(1+\rho\lambda_b)
}
|\widetilde Y_{ab}|^2
\nonumber\\
&=
Cq\nlsum_{\rho=\pm1}
\Tr(
H_\rho^{-1}B_t^{2q-2}YH_\rho^{-1}Y
),
\label{eq:tail-Hessian-upper}
\end{align}
and therefore, 
\begin{equation}
\label{eq:hessiantail}
    \E_{\mu_t}D^2\Tr r_q(B_t)[Y,Y]
\le
Cq\nlsum_{\rho=\pm1}
\mathcal Q^{(t)}_{\rho\rho}(B_t^{q-1}Y).
\end{equation}

For $t\in[0,1]$, let $\Phi_t$ denote the map  which fixes
the diagonal entries of a matrix and multiplies its off-diagonal
entries by $t$, so that $M_i(t)=\Phi_t(M_i)$.
The map $\Phi_t$ is unital and completely positive, and hence
Kadison's inequality gives $M_i(t)^2
= \Phi_t(M_i)^2
\preceq
\Phi_t(M_i^2)$, and hence

\begin{equation}
\nlsum_{i=1}^n M_i(t)^2
\preceq
\Phi_t\Bigl(\nlsum_{i=1}^n M_i^2\Bigr)
\preceq
\eta I_d.
\label{eq:path-variance}
\end{equation}

Moreover, since $M_{i,\mathrm d}=\diag(M_i)$ and $M_i$ is symmetric,
we have $M_{i,\mathrm d}^2
\preceq
\diag(M_i^2).$
Indeed, for every $1\le j\le d$, $(M_i)_{jj}^2
\le
\nlsum_{k=1}^d(M_i)_{jk}^2
=
(M_i^2)_{jj}.$
Consequently,
\begin{equation}
\nlsum_{i=1}^nM_{i,\mathrm d}^2
\preceq
\diag\Bigl(\nlsum_{i=1}^nM_i^2\Bigr)
\preceq
\eta I_d.
\label{eq:diagonal-variance}
\end{equation}
Using the matrix inequality 
\[
(X-Y)^2\preceq2X^2+2Y^2 \qquad \text{for} \; X,Y\in \Sym_d(\R),
\]
we obtain,  
\begin{equation}
\nlsum_{i=1}^nM_{i,\mathrm o}^2
\preceq
2\nlsum_{i=1}^nM_i^2
+
2\nlsum_{i=1}^nM_{i,\mathrm d}^2
\preceq
4\eta I_d.
\label{eq:offdiag-variance}
\end{equation}

Both variance bounds in \eqref{eq:path-variance} and \eqref{eq:offdiag-variance}  hold uniformly in
$t\in[0,1]$, and therefore we obtain using Lemma \ref{lem:polynomial-two-slot} that

\begin{equation}
\label{eq:boundcurv}
   \mathcal Q^{(t)}_{\rho\rho}(Y)
\le
Cd\eta,
\end{equation}
and
\begin{equation}
  \label{eq:boundcurv2}
   \qquad \mathcal Q^{(t)}_{\rho\rho}(B_t^{q-1}Y)
\le
d(Cq^3\eta)^q. 
\end{equation}
Finally, substituting  \eqref{eq:boundcurv} and \eqref{eq:boundcurv2} into \eqref{eq:hessianlog} and \eqref{eq:hessiantail}, respectively, concludes the proof.
\end{proof}

\subsection{Proof of Theorem \ref{thm:mesoscopic-smallball}}

Let $c_*>0$ be some constant. Let  $1\le n\le d$ and $A_1,\dots,A_n$ symmetric $d \times d$ matrices  with $\opnorm{A_i}\le1$ and let $L=1+\log\frac{2d}{n},
\,
R=\kappa\sqrt n\,L^2$.
Rescale the matrices $A_i$ by setting $M_i=\frac{A_i}{R}$. Let $B(g)=\nlsum_{i=1}^n g_iM_i$.  Since $\norm{A_i}_{\text{op}}\le1$, we have
$\nlsum_{i=1}^nM_i^2
\preceq
\eta I_d$ and $\eta:=\frac1{\kappa^2L^4}.$ Moreover,   denote with $A_*,D_*$ and $c_{\mathrm{curv}}$ the universal constants appearing in Proposition \ref{prop:two-slot-curvature}, and let $c_P$ be the absolute constant as defined  in Proposition \ref{thm:defective-two-slot}.
Likewise, let $C_{\mathrm I}\ge2$ be the universal constant that appears in \eqref{eq:expected-interpolation-hessian} in
Lemma~\ref{lem:expected-interpolation-hessian}.

 Let $\delta=\frac14\min\{c_*,1\}.$
We choose the constants in the following order.  First fix a universal
constant $D\ge\max\{D_*,1\}.$
Next choose $A=A(\delta)$ sufficiently large so that
\begin{equation}
A\ge A_*,
\qquad
A\ge5,
\qquad
A^2\ge2D,
\qquad
\frac{C_{\mathrm I}D}{A}\le\delta.
\label{eq:A-choice}
\end{equation}
We then choose $c_q=c_q(\delta)\ge2$ sufficiently large that
\begin{equation}
2Ae^2\,4^{-c_q}\le\delta.
\label{eq:cq-choice}
\end{equation}
Let $r=\frac dn$ and $L=1+\log(2r)$, and  set
\begin{equation}
q=\lceil c_qL\rceil,
\qquad
\alpha=A\max\{1,r\eta\},
\qquad
\beta=\frac D\alpha.
\label{eq:final-parameter-choice}
\end{equation}

We now verify that our choice of parameters satisfies,
\begin{equation}
\alpha\ge\beta, \qquad
\alpha\ge A_*,
\qquad
\alpha\beta\ge D_*,
\qquad
q^2\eta\le c_{\mathrm{P}}.
\label{eq:curvature-parameter-conditions-final}
\end{equation}

By construction, $\alpha\ge A\ge A_*$ and $\alpha\beta=D\ge D_*.$
Moreover, since $\alpha\ge A$ and $A^2\ge2D$, we have $\alpha^2\ge2D=2\alpha\beta,$
so that $\alpha\ge2\beta.$

Since $c_qL\ge2$, we have $q=\lceil c_qL\rceil\le2c_qL.$
Consequently,
\begin{equation}
q^2\eta
\le
\frac{4c_q^2L^2}{\kappa^2L^4}
\le
\frac{4c_q^2}{\kappa^2}.
\label{eq:final-q2eta}
\end{equation}
Thus, after increasing $\kappa$ if necessary, we may  assume $q^2\eta\le c_{\mathrm{P}}.$

Hence our choice of  parameters satisfies \eqref{eq:curvature-parameter-conditions-final}.
We next would like to estimate $C_{\mathrm I}\beta d\eta
+
C_{\mathrm I}\alpha qd
\bigl(C_{\mathrm I}q^3\eta\bigr)^q$ in terms of the parameters we have specified.
To this end, using,  $d=rn$ we obtain,
\begin{equation}
    C_{\mathrm I}\beta d\eta = C_{\mathrm I}
\frac{D}{A\max\{1,r\eta\}}
rn\eta \le\frac{C_{\mathrm I}D}{A}\,n
\le
\delta n. \label{eq:final-floor-bound}
\end{equation}
Moreover,  after increasing $\kappa$ so that
$\eta\le1$, we obtain
\begin{equation}
\alpha r
=
Ar\max\{1,r\eta\}
\le
A\max\{r,r^2\eta\}
\le
Ar^2
\le
Ae^{2L}.
\label{eq:final-alpha-r}
\end{equation}
Now, set $\theta=\frac{8C_{\mathrm I}c_q^3}{\kappa^2}$. For $\kappa$ sufficiently large we have $\theta\le1$.  Since
 $C_{\mathrm I}q^3\eta
\le
\frac{8C_{\mathrm I}c_q^3}{\kappa^2L}$ and $q\ge c_qL$, using 
\eqref{eq:final-alpha-r} we obtain,
\begin{align}
\frac1n
C_{\mathrm I}\alpha qd
\bigl(C_{\mathrm I}q^3\eta\bigr)^q
&=
C_{\mathrm I}\alpha qr
\bigl(C_{\mathrm I}q^3\eta\bigr)^q
\nonumber\\
&\le
2C_{\mathrm I}Ac_q
L e^{2L}
\Bigl(\frac{\theta}{L}\Bigr)^q
\nonumber\\
&\le
2C_{\mathrm I}Ac_q
L e^{2L}\theta^{c_qL}.
\label{eq:final-tail-intermediate}
\end{align}
The right-hand side of \eqref{eq:final-tail-intermediate} tends to zero uniformly for $L\ge1$ as
$\kappa\to\infty$.  Indeed, if $x=e^2\theta^{c_q}$
is chosen so that $x\le e^{-1}$, then
$Lx^L\le x$ for every $L\ge1$.  We may therefore increase $\kappa$
until
\[
2C_{\mathrm I}Ac_q x\le\delta,
\]
which then finally gives
\begin{equation}
C_{\mathrm I}\alpha qd
\bigl(C_{\mathrm I}q^3\eta\bigr)^q
\le
\delta n.
\label{eq:final-tail-bound}
\end{equation}

Let $F(t)=\log Z_t$ be the logarithm of the partition function along
the interpolation as defined in  Section~\ref{sec:interpolation}. Now, by
Lemma~\ref{lem:expected-interpolation-hessian},
\eqref{eq:final-floor-bound}, and
\eqref{eq:final-tail-bound}, we obtain uniformly in $t\in[0,1]$
\begin{equation}
F''(t)\ge-2\delta n.
\label{eq:final-F-second-bound}
\end{equation}

In the remainder, our aim will be to lower-bound $F(0)$.  To this end, we write,
\[
B_0(g)
=
\diag(\zeta_1(g),\dots,\zeta_d(g)),
\]
where $\zeta=(\zeta_1,\dots,\zeta_d)\sim\mathcal N(0,\Sigma)$.
Using that $\nlsum_{i=1}^nM_i^2
\preceq
\eta I_d$, we obtain $\Var(\zeta_j)
=
\nlsum_{i=1}^n(M_i)_{jj}^2
\le
\bigl(\nlsum_{i=1}^nM_i^2\bigr)_{jj}
\le
\eta$, and consequently $\Tr\Sigma\le d\eta.$

For $|z|\le1/2$ we have $\ell(z)
=
-\log(1-z^2)
\le
2z^2,$
and $r_q(z)
=
\nlsum_{k=q}^\infty\frac{z^{2k}}k
\le
\frac{2}{q}4^{-q}.$
Since $\alpha\ge\beta$, it follows that
\begin{equation}
v_{q,\alpha,\beta}(z)
\le
2\beta z^2
+
\frac{2\alpha}{q}4^{-q},
\qquad |z|\le\frac12.
\label{eq:diagonal-potential-bound}
\end{equation}

Let $E=\Bigl\{
|\zeta_j|\le\frac12
\text{ for every }1\le j\le d
\Bigr\}.$
Using \eqref{eq:diagonal-potential-bound}, we obtain
\begin{align}
Z_0
&=
\E\exp\Bigl(
-\nlsum_{j=1}^d
v_{q,\alpha,\beta}(\zeta_j)
\Bigr)
\nonumber\\
&\ge
\exp\Bigl(
-\frac{2\alpha d}{q}4^{-q}
\Bigr)
\E\bigl[
e^{-2\beta\norm{\zeta}_2^2}\mathbf 1_E
\bigr].
\label{eq:diagonal-Z0-first}
\end{align}

Moreover, using the standard formula for the Laplace transform of a Gaussian quadratic form gives,
\begin{align}
\E e^{-2\beta\norm{\zeta}_2^2}
&=
\det(I_d+4\beta\Sigma)^{-1/2}
\nonumber\\
&\ge
\exp(-2\beta\Tr\Sigma)
\ge
\exp(-2\beta d\eta).
\label{eq:diagonal-quadratic-tilt}
\end{align}
Define the tilted probability measure $\widetilde{\mu}$ on $\R^d$ by 
\[
\frac{d\widetilde{\mu}}{d\mu_\zeta}(z)
=
\frac{e^{-2\beta\norm{z}_2^2}}
{\E e^{-2\beta\norm{\zeta}_2^2}},
\]
where $\mu_\zeta$ denotes the law of the centered Gaussian vector
$\zeta$ with covariance matrix $\Sigma$. Then $\widetilde{\mu}$ is again a centered Gaussian measure, with
covariance matrix $\widetilde{\Sigma}
=
\Sigma^{1/2}
(I_d+4\beta\Sigma)^{-1}
\Sigma^{1/2}.$
Since $(I_d+4\beta\Sigma)^{-1}\preceq I_d,$
we have $\widetilde{\Sigma}\preceq\Sigma.$
In particular,  under $\widetilde{\mu}$ each coordinate has variance 
at most $\eta$.  Applying the Gaussian correlation inequality \cite{Royen2014GaussianCorrelation,Sidak1967Rectangular} under the tilted measure
$\widetilde{\mu}$ to the symmetric strips
$\{|\zeta_j|\le1/2\}$ gives
\[
\widetilde{\mu}(E)
\ge
\nlprod_{j=1}^d
\widetilde{\mu}\{|\zeta_j|\le1/2\}
\ge
(1-2e^{-1/(8\eta)})^d.
\]
After increasing $\kappa$ once more, we may assume
$2e^{-1/(8\eta)}\le1/2$.  Since
$\log(1-x)\ge-2x$ for $0\le x\le1/2$, it follows that
\begin{equation}
\widetilde{\mu}(E)
\ge
\exp(
-4d e^{-1/(8\eta)}
).
\label{eq:diagonal-strip-exponential}
\end{equation}
Therefore, combining
\eqref{eq:diagonal-Z0-first},
\eqref{eq:diagonal-quadratic-tilt}, and
\eqref{eq:diagonal-strip-exponential}, we arrive at
\begin{equation}
F(0)
\ge
-2\beta d\eta
-
\frac{2\alpha d}{q}4^{-q}
-
4d e^{-1/(8\eta)}.
\label{eq:final-F0-three-errors}
\end{equation}

We now check that each of the three terms on the right-hand side of \eqref{eq:final-F0-three-errors} is at
most $\delta n$. To this end, note that $2\beta d\eta
\le
\frac{2D}{A}n
\le
\delta n,$
Moreover, using \eqref{eq:final-alpha-r},
$q\ge c_qL$, and $q\ge1$,
\begin{align}
\frac1n
\frac{2\alpha d}{q}4^{-q}
&=
\frac{2\alpha r}{q}4^{-q}
\nonumber\\
&\le
2Ae^{2L}4^{-c_qL}
\nonumber\\
&=
2A(e^2 4^{-c_q})^L
\le
2Ae^2 4^{-c_q}
\le
\delta
\label{eq:final-diagonal-tail}
\end{align}
by the choice of $c_q$ in \eqref{eq:cq-choice}.

Finally, for $\kappa\ge4$ and $L\ge1$, using $L-\frac{\kappa^2L^4}{8}
\le
-\frac{\kappa^2L^4}{16}
\le
-\frac{\kappa^2}{16}$, we obtain
\begin{equation}
\frac{4d}{n}e^{-1/(8\eta)}
=
4r
\exp\Bigl(
-\frac{\kappa^2L^4}{8}
\Bigr)
\le
4\exp\Bigl(
L-\frac{\kappa^2L^4}{8}
\Bigr)\le 4\exp\Bigl(-\frac{\kappa^2}{16}
\Bigr)
\label{eq:final-strip-pre}
\end{equation}

Thus, by increasing $\kappa$ such that $4e^{-\kappa^2/16}\le\delta,$
we obtain $4d e^{-1/(8\eta)}
\le
\delta n.$
Therefore, we have by 
\eqref{eq:final-F0-three-errors} that
\begin{equation}
F(0)\ge-3\delta n.
\label{eq:final-F0}
\end{equation}

Now, note that we have  $F'(0)=0.$
Indeed, $F'(t)
=
-\E_{\mu_t}
D\mathcal V(B_t)[Y].$
At $t=0$, the matrix $B_0$ is diagonal, and therefore
$v'_{q,\alpha,\beta}(B_0)$ is diagonal, whereas $Y=B_{\mathrm o}$
has zero diagonal.  Hence $D\mathcal V(B_0)[Y]
=
\Tr(
v'_{q,\alpha,\beta}(B_0)Y
)
=
0.$

Finally, using Taylor's formula together with
\eqref{eq:final-F-second-bound} and
\eqref{eq:final-F0}, now gives
\begin{align}
F(1)
&=
F(0)+F'(0)
+
\int_0^1(1-t)F''(t)\,dt
\nonumber\\
&\ge
-3\delta n
-
2\delta n\int_0^1(1-t)\,dt
\nonumber\\
&=
-4\delta n
\ge
-c_*n,
\label{eq:final-F1}
\end{align}
where the last inequality follows from our choice of $\delta=\frac14\min\{c_*,1\}$.
By definition of the potential, we have 
$e^{-\mathcal V(B)}=0$ whenever $\opnorm{B}\ge1$, while
$0\le e^{-\mathcal V(B)}\le1$ on $\{\opnorm{B}<1\}$.  Consequently,
\[
\Prob\{\opnorm{B}<1\}
\ge
\E e^{-\mathcal V(B)}
=
Z_1
=
e^{F(1)}
\ge
e^{-c_*n}.
\]
But, since $B=R^{-1}\nlsum_i g_iA_i$ and
$R=\kappa\sqrt n\,L^2$, we therefore conclude that
\[
\Prob\biggl\{
\opnorm[\Big]{\nlsum_{i=1}^ng_iA_i}
<
\kappa\sqrt n
\Bigl(1+\log\frac{2d}{n}\Bigr)^2
\biggr\}
\ge
e^{-c_*n},
\]
which proves
Theorem~\ref{thm:mesoscopic-smallball}.

Finally, note all lower bounds imposed on $\kappa$ depend only on
$\delta$, $A$, $c_q$, and other universal constants.  Since
$\delta$, $A$, and $c_q$ depend only on $c_*$, the resulting
$\kappa=\kappa(c_*)$ is independent of $n$, $d$, and the matrices
$A_1,\dots,A_n$.

We also record where the matrix dimension $d$ entered the argument: only through the trace of the matrix variance, $\Tr\nlsum_iM_i^2\le d\eta$, in \eqref{eq:final-floor-bound} and \eqref{eq:final-F0-three-errors}, and through the dimensional factor in the Gaussian trace-moment bound of Lemma~\ref{lem:gaussian-trace-moment}, in \eqref{eq:final-tail-bound}. Both were absorbed by the choice of parameters, at the price of precisely the polylogarithmic factor $L^2=(1+\log\frac{2d}{n})^2$ in the radius $R$.

\subsection{Proof of Matrix Spencer}

\label{sec:matrix-spencer-completion}

We now combine Theorem~\ref{thm:mesoscopic-smallball} with Gluskin's method for partial coloring. We use the shifted version due to Reis and Rothvo{\ss} \cite[Theorem~6]{ReisRothvoss2023}, in the form stated in \cite[Theorem~3.1]{dadush2022matrix}.
\begin{theorem}[Algorithmic Gaussian partial coloring]
\label{thm:algorithmic-partial-coloring}
For every fixed $\theta>0$ there exists a constant
$C_{\mathrm{pc}}=C_{\mathrm{pc}}(\theta)>0$ with the following
property. Let $K\subseteq\R^m$ be a symmetric convex body satisfying
\[
\gamma_m(K)\ge 2^{-\theta m}.
\]
Then, given a current fractional coloring
$y\in(-1,1)^m$, there is a randomized polynomial-time algorithm
which finds an increment $x\in C_{\mathrm{pc}}K$ such that
\[
y+x\in[-1,1]^m
\]
and
\[
\bigl|
\{
i\in[m]:
|y_i+x_i|=1
\}
\bigr|
\ge \frac m2.
\]
\end{theorem}

\begin{proof}[Proof of Theorem~\ref{thm:matrix-spencer}]
Let $A_1,\dots,A_n\in\Sym_n(\R),$ satisfying $\opnorm{A_i}\le1.$
Let $c_*=\log 2$ and let $\kappa_0:=\kappa(\log2).$
We construct the coloring iteratively.  Let $y^{(0)}=0\in[-1,1]^n.$
At the beginning of the $k$-th round, let $S_k
=
\{
i\in[n]:
|y_i^{(k)}|<1
\}$
be the set of active coordinates, and let $m_k=|S_k|.$
If $m_k=0$, then our iteration terminates and we have fully colored all coordinates. 
For an active set $S\subseteq[n]$ of cardinality $m$, define
\begin{equation}
K_S
=
\Bigl\{
x\in\R^S:
\opnorm[\Big]{\nlsum_{i\in S}x_iA_i}
\le
R(m)
\Bigr\},
\label{eq:active-discrepancy-body}
\end{equation}
where
\begin{equation}
R(m)
=
\kappa_0\sqrt m
\Bigl(
1+\log\frac{2n}{m}
\Bigr)^2.
\label{eq:active-radius}
\end{equation}
The set $K_S$ is a symmetric convex  subset of $\R^S$. Applying Theorem~\ref{thm:mesoscopic-smallball} to the family
$(A_i)_{i\in S}$, with matrix dimension $d=n$ and
number of coefficient matrices being $m=|S|$, gives $\gamma_m(K_S)
\ge
e^{-(\log2)m}
=
2^{-m}.$
Thus Theorem~\ref{thm:algorithmic-partial-coloring} applies with
$\theta=1$.  Let $C_{\mathrm{pc}}:=C_{\mathrm{pc}}(1)$.
Applying the partial-coloring algorithm to $K_{S_k}$ and the current
fractional coloring $y^{(k)}|_{S_k}\in(-1,1)^{S_k}$, produces an increment $x^{(k)}\in C_{\mathrm{pc}}K_{S_k}$
such that $y^{(k)}|_{S_k}+x^{(k)}
\in[-1,1]^{S_k}$
and at least half of the coordinates in $S_k$ become equal to
$\pm1$. For the next iteration, extend $x^{(k)}$ by zero outside $S_k$ and set $y^{(k+1)}=y^{(k)}+x^{(k)}$, and note that previously frozen coordinates are unchanged.  Moreover,
\begin{equation}
m_{k+1}\le\frac12m_k.
\label{eq:active-halving}
\end{equation}

Since $x^{(k)}\in C_{\mathrm{pc}}K_{S_k}$, the discrepancy increment
at the $k$-th round satisfies
\begin{equation*}
\opnorm[\Big]{
\nlsum_{i\in S_k}x_i^{(k)}A_i
}
\le
C_{\mathrm{pc}}\kappa_0
\sqrt{m_k}
\Bigl(
1+\log\frac{2n}{m_k}
\Bigr)^2.
\end{equation*}

After at most $\lceil\log_2n\rceil+1$ rounds there are no active
coordinates, so the final vector $\varepsilon:=y^{(T)}$
belongs to $\{\pm1\}^n$.  Since $y^{(0)}=0$, $\varepsilon
=
\nlsum_{k=0}^{T-1}x^{(k)},$
and hence
\begin{align}
\opnorm[\Big]{
\nlsum_{i=1}^n\varepsilon_iA_i
}
&\le
\nlsum_{k=0}^{T-1}
\opnorm[\Big]{
\nlsum_{i\in S_k}x_i^{(k)}A_i
}
\nonumber\\
&\le
C_{\mathrm{pc}}\kappa_0
\nlsum_{k=0}^{T-1}
\sqrt{m_k}
\Bigl(
1+\log\frac{2n}{m_k}
\Bigr)^2.
\label{eq:matrix-spencer-before-sum}
\end{align}

It remains to show that the radii in
\eqref{eq:matrix-spencer-before-sum} are summable. To this end, for each nonzero
$m_k$, define $j_k
=
\Bigl\lfloor
\log_2\frac{n}{m_k}
\Bigr\rfloor.$
By \eqref{eq:active-halving}, we have  $j_{k+1}\ge j_k+1$. Moreover, $2^{j_k}
\le
\frac{n}{m_k}
<
2^{j_k+1},$
so that $m_k\le 2^{-j_k}n$ and $1+\log\frac{2n}{m_k}< 1+(j_k+2)\log2\le 3(1+j_k).$
Therefore,
\begin{align*}
\sqrt{m_k}
\Bigl(
1+\log\frac{2n}{m_k}
\Bigr)^2
&\le
9\sqrt n\,
2^{-j_k/2}(1+j_k)^2.
\end{align*}
Since the integers $j_k$ are strictly increasing, we obtain
\begin{align}
\nlsum_{k=0}^{T-1}
\sqrt{m_k}
\Bigl(
1+\log\frac{2n}{m_k}
\Bigr)^2
&\le
9\sqrt n
\nlsum_{j=0}^{\infty}
2^{-j/2}(1+j)^2
\nonumber\\
&\le
C_{\mathrm{sum}}\sqrt n,
\label{eq:hereditary-radius-sum}
\end{align}
where $C_{\mathrm{sum}}
=
9\nlsum_{j=0}^{\infty}
2^{-j/2}(1+j)^2
<\infty$.
Finally, substituting \eqref{eq:hereditary-radius-sum} into
\eqref{eq:matrix-spencer-before-sum} gives
\[
\opnorm[\Big]{
\nlsum_{i=1}^n\varepsilon_iA_i
}
\le
C_{\mathrm{pc}}\kappa_0C_{\mathrm{sum}}\sqrt n.
\]
Thus, with $C
=
C_{\mathrm{pc}}\kappa_0C_{\mathrm{sum}},$
we obtain
\[
\opnorm[\Big]{
\nlsum_{i=1}^n\varepsilon_iA_i
}
\le
C\sqrt n,
\]
which  proves Theorem~\ref{thm:matrix-spencer}.
\end{proof}

\section*{Acknowledgments}
We would like to thank Afonso Bandeira for interesting discussions on Matrix Spencer and for highlighting matrix small-ball. SS also thanks a series of previous students and collaborators with whom he has previously discussed the Matrix Spencer conjecture.

\section*{Statement on AI Usage}
This work made extensive use of GPT-5.6 Sol (Pro) over a period of months for exploring approaches to the Matrix Spencer Conjecture. We trace back the starting point of the full resolution to a suggestion by the second named author (SS) to the AI assistant to follow the route of using log-barriers to soften the operator-norm based region in the Gluskin-style matrix small ball estimates (Gaussian measure of a suitable convex set), and the aim of using Schur complements to simplify the argumentation of Theorem \ref{thm:rsc}. Subsequently, GPT found a full proof after some failed attempts and a lot of back-n-forth with SS for the hereditary small-ball theorem---we attribute the heavy-lifting for that lemma and for almost all the calculations presented in this writeup to GPT. After the hereditary small-ball lemma we had all the pieces for the proof of the conjecture. At that point we took GPT's help for preparing the first rough draft of the work, which was then subsequently polished by EA and SS to bring it in the present form. We have also simplified the proofs, verified them, and cleaned up exposition. The responsibility for errors and omissions, of course, lies with the authors.

\bibliographystyle{plainnat}

\appendix
\section{Proofs of some technical statements from Section \ref{sec:matrixsmall}}
\subsection{Proof of Lemma \ref{lem:matrix-weighted-curvature}}
\label{app:curvature-criterion}

For an open set $\Omega'\subseteq\Omega$ we write
\[
\ip{F}{G}_{W,\Omega'}=\int_{\Omega'}F^\top WG\,dx,
\qquad
\mathcal E_{\Omega'}(F,G)=\int_{\Omega'}\nlsum_{i=1}^n(\partial_iF)^\top W(\partial_iG)\,dx,
\]
for the weighted inner product and the weighted Dirichlet form, we set
$\norm{F}_{W,\Omega'}^2=\ip{F}{F}_{W,\Omega'}$, and we let
\[
\mathcal L=-\nlsum_{i=1}^n\bigl(\partial_i^2+W^{-1}(\partial_iW)\partial_i\bigr),
\qquad\text{so that}\qquad
W\mathcal LG=-\nlsum_{i=1}^n\partial_i\bigl(W\partial_iG\bigr).
\]
The proof proceeds in three steps. We first establish a Bochner--Reilly
inequality on smooth bounded convex subdomains $\Omega'$ with
$\overline{\Omega'}\subset\Omega$ for fields satisfying a Neumann boundary
condition; we then deduce the Poincar\'e inequality on such subdomains;
finally we exhaust $\Omega$. Since $W$ is smooth and uniformly positive
definite on $\overline{\Omega'}$, no question concerning the degeneration
of $W$ near $\partial\Omega$ arises. In the scalar case this is the
classical route to Brascamp--Lieb-type inequalities on convex domains via
Reilly's formula \cite{Reilly1977}; see, e.g., \cite{KolesnikovMilman2017}.

\paragraph{Step 1: A Bochner--Reilly inequality.}
Let $\Omega'$ be a bounded convex domain with $C^\infty$ boundary such
that $\overline{\Omega'}\subset\Omega$, let $\nu$ denote its outward unit
normal, and let $G\in C^\infty(\overline{\Omega'};\R^r)$ satisfy the
Neumann condition $\partial_\nu G=0$ on $\partial\Omega'$. We claim that
\begin{align}
\norm{\mathcal LG}_{W,\Omega'}^2
={}&
\nlsum_{i,j=1}^n\norm{\partial_{ij}G}_{W,\Omega'}^2
+
\int_{\Omega'}\nlsum_{i,j=1}^n(\partial_iG)^\top\Theta_{ij}(W)(\partial_jG)\,dx
\nonumber\\
&+
\int_{\partial\Omega'}\nlsum_{a,b=1}^{n-1}S_{ab}\,(\partial_{\tau_a}G)^\top W(\partial_{\tau_b}G)\,dS,
\label{eq:matrix-bochner}
\end{align}
where, at each point of $\partial\Omega'$, $\tau_1,\dots,\tau_{n-1}$ is
an orthonormal basis of the tangent space and
$S_{ab}=\ip{D_{\tau_a}\nu}{\tau_b}$ is the second fundamental form of
$\partial\Omega'$. Since $\Omega'$ is convex, $(S_{ab})_{a,b}\succeq0$, and
since $W\succ0$, the matrix $M_{ab}=(\partial_{\tau_a}G)^\top W(\partial_{\tau_b}G)$
is a Gram matrix; hence the boundary integrand $\Tr(SM)$ is nonnegative.
Together with the curvature assumption \eqref{eq:abstract-curvature-gap},
\eqref{eq:matrix-bochner} therefore yields
\begin{equation}
\norm{\mathcal LG}_{W,\Omega'}^2
\ge
\lambda\,\mathcal E_{\Omega'}(G,G)
=
\lambda\ip{G}{\mathcal LG}_{W,\Omega'},
\label{eq:bochner-spectral-gap}
\end{equation}
where the last identity follows from \eqref{eq:neumann-ibp} below (with
$X=\partial_iG$, $Y=G$) and $\partial_\nu G=0$.

To prove \eqref{eq:matrix-bochner}, write $\partial_i^*=-\partial_i-W^{-1}(\partial_iW)$,
so that $\mathcal L=\nlsum_i\partial_i^*\partial_i$. Integration by parts
gives, for all $X,Y\in C^1(\overline{\Omega'};\R^r)$,
\begin{equation}
\ip{\partial_i^*X}{Y}_{W,\Omega'}
=
\ip{X}{\partial_iY}_{W,\Omega'}
-
\int_{\partial\Omega'}\nu_i\,X^\top WY\,dS.
\label{eq:neumann-ibp}
\end{equation}
Moreover, a direct computation shows that
\[
[\partial_j,\partial_i^*]Y
=
-\partial_j\bigl(W^{-1}\partial_iW\bigr)Y
=
W^{-1}\Theta_{ji}(W)Y .
\]
Applying \eqref{eq:neumann-ibp} with $X=\mathcal LG$ and $Y=\partial_jG$,
summing over $j$, and using $\partial_\nu G=0$, we obtain
\[
\norm{\mathcal LG}_{W,\Omega'}^2
=
\nlsum_{j=1}^n\ip{\mathcal LG}{\partial_j^*\partial_jG}_{W,\Omega'}
=
\nlsum_{j=1}^n\ip{\partial_j\mathcal LG}{\partial_jG}_{W,\Omega'}.
\]
Next, $\partial_j\mathcal LG=\nlsum_i\partial_i^*\partial_{ij}G+\nlsum_iW^{-1}\Theta_{ji}(W)\partial_iG$,
and applying \eqref{eq:neumann-ibp} once more, now with $X=\partial_{ij}G$
and $Y=\partial_jG$, together with $\Theta_{ji}(W)^\top=\Theta_{ij}(W)$,
yields
\begin{align*}
\norm{\mathcal LG}_{W,\Omega'}^2
={}&
\nlsum_{i,j=1}^n\norm{\partial_{ij}G}_{W,\Omega'}^2
+
\int_{\Omega'}\nlsum_{i,j=1}^n(\partial_iG)^\top\Theta_{ij}(W)(\partial_jG)\,dx\\
&-
\int_{\partial\Omega'}\nlsum_{j=1}^n(\partial_\nu\partial_jG)^\top W(\partial_jG)\,dS.
\end{align*}
It remains to identify the boundary term. Fix a point of $\partial\Omega'$
and the orthonormal basis $\nu,\tau_1,\dots,\tau_{n-1}$ of $\R^n$. The sum
$\nlsum_j(D^2G[\nu,e_j])^\top W(DG[e_j])$ is invariant under orthonormal
changes of basis, and $DG[\nu]=\partial_\nu G=0$; hence
\[
\nlsum_{j=1}^n(\partial_\nu\partial_jG)^\top W(\partial_jG)
=
\nlsum_{a=1}^{n-1}\bigl(D^2G[\nu,\tau_a]\bigr)^\top W\bigl(DG[\tau_a]\bigr).
\]
Extend $\nu$ to a smooth unit vector field on a neighbourhood of
$\partial\Omega'$. Since $DG[\nu]=0$ on $\partial\Omega'$, differentiating
in the tangential direction $\tau_a$ gives
$D^2G[\tau_a,\nu]+DG[D_{\tau_a}\nu]=0$ on $\partial\Omega'$, where
$D_{\tau_a}\nu=\nlsum_bS_{ab}\tau_b$ is tangential. Therefore
\[
\nlsum_{a=1}^{n-1}\bigl(D^2G[\nu,\tau_a]\bigr)^\top W\bigl(DG[\tau_a]\bigr)
=
-\nlsum_{a,b=1}^{n-1}S_{ab}\,(\partial_{\tau_b}G)^\top W(\partial_{\tau_a}G),
\]
which, by the symmetry of $S$ and $W$, proves \eqref{eq:matrix-bochner}.

\paragraph{Step 2: The Poincar\'e inequality on $\Omega'$.}
Let $\Omega'$ be as in Step 1. Since $\overline{\Omega'}\subset\Omega$ is
compact, $W$ is $C^\infty$ and uniformly positive definite on
$\overline{\Omega'}$, and $\mathcal E_{\Omega'}$ is the bilinear form of
the strongly elliptic system $G\mapsto-\nlsum_i\partial_i(W\partial_iG)$
with $C^\infty$ coefficients, whose conormal derivative is
$W\partial_\nu G$. Let $\Phi\in C^\infty(\overline{\Omega'};\R^r)$ satisfy
$\int_{\Omega'}W\Phi\,dx=0$. On the closed subspace
$\{G\in H^1(\Omega';\R^r):\int_{\Omega'}WG\,dx=0\}$ the form
$\mathcal E_{\Omega'}$ is coercive, by the Poincar\'e--Wirtinger
inequality on the bounded Lipschitz domain $\Omega'$. Hence, by the
Lax--Milgram theorem, there is $G\in H^1(\Omega';\R^r)$ with
$\mathcal E_{\Omega'}(G,\varphi)=\ip{\Phi}{\varphi}_{W,\Omega'}$ for all
$\varphi$ in this subspace, and then for all $\varphi\in H^1(\Omega';\R^r)$,
since both sides vanish for constant $\varphi$. That is, $G$ is a weak
solution of the Neumann problem
$-\nlsum_i\partial_i(W\partial_iG)=W\Phi$ in $\Omega'$, $W\partial_\nu G=0$
on $\partial\Omega'$. By elliptic regularity up to the boundary for
strongly elliptic systems with smooth coefficients on smooth domains
\cite[Chapter~4]{McLean2000} (see also \cite{AgmonDouglisNirenberg1964}),
$G\in C^\infty(\overline{\Omega'};\R^r)$, so that $\mathcal LG=\Phi$ in
$\Omega'$ and $\partial_\nu G=0$ on $\partial\Omega'$.

Now let $F\in H^1(\Omega';\R^r)$ with $\int_{\Omega'}WF\,dx=0$, and let
$\Phi$ and $G$ be as above. Integrating by parts (using
$\partial_\nu G=0$), and using the Cauchy--Schwarz inequality for the
nonnegative form $\mathcal E_{\Omega'}$ together with
\eqref{eq:bochner-spectral-gap}, we obtain
\begin{align*}
|\ip{F}{\Phi}_{W,\Omega'}|
=
|\ip{F}{\mathcal LG}_{W,\Omega'}|
=
|\mathcal E_{\Omega'}(F,G)|
&\le
\mathcal E_{\Omega'}(F,F)^{1/2}\,\mathcal E_{\Omega'}(G,G)^{1/2}\\
&\le
\lambda^{-1/2}\,\mathcal E_{\Omega'}(F,F)^{1/2}\,\norm{\Phi}_{W,\Omega'} .
\end{align*}
Smooth fields $\Phi$ with $\int_{\Omega'}W\Phi\,dx=0$ are dense in
$\{\Phi\in L^2(\Omega',W\,dx):\int_{\Omega'}W\Phi\,dx=0\}$, a space which
contains $F$; choosing $\Phi\to F$ in $L^2(\Omega',W\,dx)$ therefore gives
$\norm{F}_{W,\Omega'}^2\le\lambda^{-1}\mathcal E_{\Omega'}(F,F)$. Applying
this to $F-\bar F$, where
$\bar F=(\int_{\Omega'}W\,dx)^{-1}\int_{\Omega'}WF\,dx$, we conclude that
every $F\in H^1(\Omega';\R^r)$ satisfies
\begin{equation}
\int_{\Omega'}F^\top WF\,dx
-
\Bigl(\int_{\Omega'}WF\,dx\Bigr)^\top
\Bigl(\int_{\Omega'}W\,dx\Bigr)^{-1}
\Bigl(\int_{\Omega'}WF\,dx\Bigr)
\le
\lambda^{-1}\mathcal E_{\Omega'}(F,F).
\label{eq:poincare-subdomain}
\end{equation}

\paragraph{Step 3: Exhaustion.}
For $k\in\mathbb N$ let
$K_k=\{x\in\Omega:\norm{x}_2\le k,\ \operatorname{dist}(x,\R^n\setminus\Omega)\ge1/k\}$
(with the convention $\operatorname{dist}(x,\emptyset)=\infty$). These are
compact convex subsets of $\Omega$ with $\bigcup_kK_k=\Omega$, and they
have nonempty interior for all large $k$. Fix such a $k$ and choose
$\rho_k>0$ with $K_k+2\rho_kB_2^n\subset\Omega$. Since convex bodies with
$C^\infty$ boundary and positive Gauss curvature are dense in the space of
convex bodies with respect to the Hausdorff metric
\cite[Section~3.4]{Schneider2014}, there is such a body $C_k$ whose
Hausdorff distance to $K_k+\rho_kB_2^n$ is less than $\rho_k/2$. Comparing
support functions, $K_k+\frac{\rho_k}{2}B_2^n\subseteq C_k\subseteq K_k+2\rho_kB_2^n$.
Hence $\Omega_k:=\operatorname{int}C_k$ is a bounded convex domain with
$C^\infty$ boundary, $\overline{\Omega_k}\subset\Omega$ and
$K_k\subseteq\Omega_k$; in particular $\one_{\Omega_k}\to\one_\Omega$
pointwise as $k\to\infty$.

Let $F\in C^1(\Omega;\R^r)$ with $\int_\Omega F^\top WF\,dx<\infty$. Then
$F|_{\Omega_k}\in C^1(\overline{\Omega_k};\R^r)\subseteq H^1(\Omega_k;\R^r)$,
so \eqref{eq:poincare-subdomain} holds on each $\Omega_k$, and
$\mathcal E_{\Omega_k}(F,F)\le\mathcal E_\Omega(F,F)$. The functions
$F^\top WF$, $\norm{W}_{\mathrm{op}}$ and
$\norm{WF}_2\le\norm{W}_{\mathrm{op}}^{1/2}(F^\top WF)^{1/2}$ are
integrable over $\Omega$, by assumption and by the Cauchy--Schwarz
inequality. Dominated convergence therefore gives
$\int_{\Omega_k}F^\top WF\,dx\to\int_\Omega F^\top WF\,dx$,
$\int_{\Omega_k}WF\,dx\to\int_\Omega WF\,dx$, and
$\int_{\Omega_k}W\,dx\to\int_\Omega W\,dx\succ0$. Letting $k\to\infty$ in
\eqref{eq:poincare-subdomain} proves \eqref{eq:abstract-matrix-poincare}.

\subsection{Proof of Lemma \ref{lem:cofactor-poincare}}

On the interior of $\Omega$, write 
\begin{equation} K(x)=\det(H(x))^4H(x)^{-1}. \label{eq:cofactor-weight-K} \end{equation}
Since $d\gamma_n$ is a constant multiple of $e^{-\norm{x}^2/2}dx$,
inequality \eqref{eq:cofactor-poincare} is precisely
\eqref{eq:abstract-matrix-poincare} for the weight
$W(x)=e^{-\norm{x}^2/2}K(x)$ on the nonempty open convex set
$\Omega=\{x:H(x)\succ0\}$, with $\lambda=1$. We verify the hypotheses of
Lemma~\ref{lem:matrix-weighted-curvature}. Clearly, $W$ is $C^\infty$ and
positive definite on $\Omega$, and $\int_\Omega W\,dx<\infty$ because
$K=\det(H)^3\operatorname{adj}(H)$ is polynomially bounded. It remains to
verify the curvature condition \eqref{eq:abstract-curvature-gap}.
To this end, we introduce the following matrix-weighted analogue of the
Hessian curvature term appearing in the classical weighted Bochner identity. For a smooth positive-definite matrix weight $M$, we define its curvature
block matrix by
\begin{equation}
\Theta(M) = [\Theta_{ij}(M)]_{ij}^n, \qquad
\Theta_{ij}(M)
=
(\partial_iM)M^{-1}(\partial_jM)-\partial_{ij}M,
\qquad 1\le i,j\le n.
\label{eq:matrix-curvature-definition}
\end{equation}
The curvature transforms particularly simply under multiplication by a
positive scalar weight. Namely, if $\rho>0$ is scalar-valued, then a direct
calculation gives
\begin{equation}
\Theta_{ij}(\rho M)
=
\rho\,\Theta_{ij}(M)
-
\rho\,\partial_{ij}\log\rho\,M.
\label{eq:scalar-curvature-transformation}
\end{equation}
We first establish that the curvature block of $K$ is positive
semidefinite. To this end, fix $x_0\in\Omega$. Since the calculation is invariant under a fixed
congruence, after conjugating by $H(x_0)^{-1/2}$ we may assume that $H(x_0)=I_r$. For ease of notation, we continue to denote the congruence-normalized coefficient matrices by
$E_i$. Direct differentiation of \eqref{eq:cofactor-weight-K} at $x_0$
gives
\begin{align}
\partial_iK
={}&
4\Tr(E_i)I_r-E_i\nonumber\\
\partial_{ij}K
={}&
-4\Tr(E_iE_j)I_r
+16\Tr(E_i)\Tr(E_j)I_r
\nonumber\\
&\quad
-4\Tr(E_i)E_j
-4\Tr(E_j)E_i
+E_jE_i+E_iE_j.
\label{eq:cofactor-second-derivative}
\end{align}
Since $K(x_0)=I_r$ under the normalization $H(x_0)=I_r$, substituting
\eqref{eq:cofactor-second-derivative} into
\eqref{eq:matrix-curvature-definition} yields
\begin{equation}
\Theta_{ij}(K)
=
4\Tr(E_iE_j)I_r-E_jE_i.
\label{eq:cofactor-curvature}
\end{equation}

We now verify positivity of the full block matrix
$[\Theta_{ij}(K)]_{i,j=1}^n$. For arbitrary
$z_1,\dots,z_n\in\R^r$, define the tensor
\begin{equation*}
T_{abc}
=
\nlsum_{i=1}^n(E_i)_{ab}(z_i)_c.
\end{equation*}
By direct expansion, we get
\begin{equation}
\nlsum_{i,j=1}^n
\Tr(E_iE_j)\ip{z_i}{z_j}
=
\nlsum_{a,b,c}T_{abc}^2
=
\norm{T}_2^2,
\label{eq:tensor-first-term}
\end{equation}
while a rearrangement of indices gives
\begin{equation}
\nlsum_{i,j=1}^n
z_i^\top E_jE_i z_j
=
\nlsum_{a,b,c}T_{abc}T_{bca}.
\label{eq:tensor-cyclic-term}
\end{equation}
The cyclic permutation
\[
(T_{abc})\longmapsto(T_{bca})
\]
is orthogonal on the underlying Euclidean tensor space. Hence,
by Cauchy--Schwarz,
\begin{equation}
\Bigl|
\nlsum_{a,b,c}T_{abc}T_{bca}
\Bigr|
\le
\norm{T}_2^2.
\label{eq:tensor-cyclic-bound}
\end{equation}
Combining equations \eqref{eq:cofactor-curvature}--
\eqref{eq:tensor-cyclic-bound}, we obtain
\begin{align}
\nlsum_{i,j=1}^n
z_i^\top\Theta_{ij}(K)z_j
&=
4\norm{T}_2^2
-
\nlsum_{a,b,c}T_{abc}T_{bca}
\nonumber\\
&\ge
3\norm{T}_2^2
\ge0,
\label{eq:cofactor-curvature-positive}
\end{align}
and therefore, we have that
\begin{equation}
[\Theta_{ij}(K)]_{i,j=1}^n\succeq0
\qquad\text{on }\Omega.
\label{eq:cofactor-curvature-block-positive}
\end{equation}
Now, we consider the matrix weighted Gaussian measure   by $W(x) =e^{-\norm{x}^2/2}K(x)$. The curvature block matrix associated with $W$ is given by
\begin{equation}
\Theta_{ij}(W)
=
e^{-\norm{x}^2/2}\Theta_{ij}(K)
+
\delta_{ij}W,
\label{eq:gaussian-curvature-decomposition}
\end{equation}
where we used that $-\partial_{ij}\log\rho=\delta_{ij}.$ Using that $[\Theta_{ij}(K)]_{i,j=1}^n\succeq0$, we obtain
\begin{equation}
[\Theta_{ij}(W)]_{i,j=1}^n
\succeq
[\delta_{ij}W]_{i,j=1}^n.
\label{eq:gaussian-curvature-gap}
\end{equation}
Thus all hypotheses of Lemma~\ref{lem:matrix-weighted-curvature} are
satisfied by $W$ with $\lambda=1$, and \eqref{eq:abstract-matrix-poincare}
is precisely \eqref{eq:cofactor-poincare}. This proves the first
assertion of Lemma~\ref{lem:cofactor-poincare}.

Finally, suppose that the weight is multiplied by a scalar
log-concave factor $e^{-U}$ with $U\in C^\infty(\Omega)$ convex. By convexity of $U$ we have $[\partial_{ij}U]_{i,j=1}^n\succeq0.$ Applying \eqref{eq:scalar-curvature-transformation} once more shows
that this multiplication adds the positive block $[(\partial_{ij}U)W]_{i,j=1}^n$
to the curvature. Hence the lower bound
\eqref{eq:gaussian-curvature-gap} remains valid, and
Lemma~\ref{lem:matrix-weighted-curvature} applies exactly as before.

\subsection{Proof of Proposition \ref{prop:S-bound}}

To bring the resolvent $G=(I-C^2)^{-1}$ within the scope of
Lemma~\ref{lem:cofactor-poincare}, we introduce the following affine linearization of $I-C^2$ 
\begin{equation}
\label{eq:H-lift}
H(g)=
\begin{pmatrix}
I&C(g)\\
C(g)&I
\end{pmatrix}.
\end{equation}
Since $C(g)=\nlsum_i g_iD_i$, \eqref{eq:H-lift} is of the form
\[
H(g)=H_0+\nlsum_i g_iE_i,
\qquad
H_0=
\begin{pmatrix}
I&0\\
0&I
\end{pmatrix},
\qquad
E_i=
\begin{pmatrix}
0&D_i\\
D_i&0
\end{pmatrix}.
\]
Moreover,
\[
H\succ0\iff\opnorm C<1,
\qquad
\det H=\det(I-C^2),
\]
and
\[
H^{-1}
=
\begin{pmatrix}
G&-CG\\
-CG&G
\end{pmatrix}.
\]
Consequently, the matrix weight in
Lemma~\ref{lem:cofactor-poincare} is
\[
\mathcal K(g)
=
\det(I-C(g)^2)^4H(g)^{-1}
\one_{\{\opnorm{C(g)}<1\}},
\]
and hence, by the definition of $\nu$,
\begin{equation}
\frac{1}{Z_C}\mathcal K(g)\,d\gamma_n(g)
=
H(g)^{-1}\,d\nu(g).
\label{eq:H-weight-nu}
\end{equation}
Thus the matrix-weighted Poincar\'e inequality translates directly
into an inequality for expectations under $\nu$ involving the resolvent
$G$.
To extract an useful inequality for $S=\E_\nu G$, we fix
$e\in\R^{d-1}$ and consider the test function
\[
F_e(g)
=
\begin{pmatrix}
C(g)e\\
0
\end{pmatrix}.
\]
Since $F_e$ is linear in $g$ and $\mathcal K$ is polynomially bounded, we have $\int F_e^\top\mathcal KF_e\,d\gamma_n<\infty$, so that Lemma~\ref{lem:cofactor-poincare} applies to $F_e$; it takes the form
\begin{align}
&\E_\nu\![F_e^\top H^{-1}F_e]
-
(\E_\nu[H^{-1}F_e])^\top
(\E_\nu H^{-1})^{-1}
(\E_\nu[H^{-1}F_e])
\nonumber\\
&\hspace{25mm}
\le
\E_\nu\!\Bigl[
\nlsum_{i=1}^n
(\partial_iF_e)^\top
H^{-1}
(\partial_iF_e)
\Bigr].
\label{eq:Fe-poincare}
\end{align}
We compute the three terms in \eqref{eq:Fe-poincare} separately. We begin with the first quadratic term on the left-hand side. Using $C^2G=G-I,$
we obtain
\begin{equation}
F_e^\top H^{-1}F_e
=
e^\top C^2Ge
=
e^\top(G-I)e,
\label{eq:Fe-norm}
\end{equation}
so that by taking expectation with respect to $\nu$ we get that
\begin{equation}
\E_\nu\![F_e^\top H^{-1}F_e]
=
e^\top(S-I)e.
\label{eq:Fe-first-term}
\end{equation}
We next compute the mean-correction term. From the block representation
of $H^{-1}$ we have,
\[
H^{-1}F_e
=
\begin{pmatrix}
G&-CG\\
-CG&G
\end{pmatrix}
\begin{pmatrix}
Ce\\
0
\end{pmatrix}
=
\begin{pmatrix}
CGe\\
-C^2Ge
\end{pmatrix}
=
\begin{pmatrix}
CGe\\
-(G-I)e
\end{pmatrix}.
\]
Since $C$ depends linearly on $g$, we have $C(-g)=-C(g)$, whereas 
\[
G(-g)
=
\bigl(I-C(-g)^2\bigr)^{-1}
=
\bigl(I-C(g)^2\bigr)^{-1}
=
G(g).
\]
Thus $CG$ is odd, while $G-I$ is even. Moreover, the density defining
$\nu$ is invariant under $g\mapsto -g$, since it depends on $g$ only
through $C(g)^2$, and therefore, we get that $\E_\nu[CG]=0$.
In particular, this means that, 
\begin{equation}
\E_\nu[H^{-1}F_e]
=
\begin{pmatrix}
0\\
-(S-I)e
\end{pmatrix}.
\label{eq:Fe-mean}
\end{equation}

The same parity argument, applied directly to $H^{-1}$, gives us
\begin{equation}
\E_\nu H^{-1}
=
\E_\nu
\begin{pmatrix}
G&-CG\\
-CG&G
\end{pmatrix}
=
\begin{pmatrix}
S&0\\
0&S
\end{pmatrix}.
\label{eq:H-inverse-mean}
\end{equation}
Therefore, by combining \eqref{eq:Fe-mean} and \eqref{eq:H-inverse-mean} we have,
\begin{align}
&(\E_\nu[H^{-1}F_e])^\top
(\E_\nu H^{-1})^{-1}
(\E_\nu[H^{-1}F_e])
\nonumber\\
&\hspace{20mm}
=
e^\top(S-I)S^{-1}(S-I)e.
\label{eq:Fe-mean-correction}
\end{align}
Finally, using \eqref{eq:Fe-first-term} and
\eqref{eq:Fe-mean-correction}, the left-hand side of
\eqref{eq:Fe-poincare} simplifies to
\begin{align}
&e^\top(S-I)e
-
e^\top(S-I)S^{-1}(S-I)e
\nonumber\\
&\hspace{20mm}
=
e^\top(I-S^{-1})e.
\label{eq:Fe-left-side}
\end{align}

It remains to compute the Dirichlet term on the right-hand side of
\eqref{eq:Fe-poincare}. To this end, we use
$\partial_iF_e
=
\begin{pmatrix}
D_ie\\
0
\end{pmatrix}$ to obtain
\begin{align}
(\partial_iF_e)^\top H^{-1}(\partial_iF_e)
&=
\begin{pmatrix}
e^\top D_i&0
\end{pmatrix}
\begin{pmatrix}
G&-CG\\
-CG&G
\end{pmatrix}
\begin{pmatrix}
D_ie\\
0
\end{pmatrix}
\nonumber\\
&=
e^\top D_iGD_i e.
\label{eq:Fe-dirichlet-pointwise}
\end{align}
Taking expectation and summing over $i$ therefore gives,
\begin{equation}
\E_\nu\!\Bigl[
\nlsum_{i=1}^n
(\partial_iF_e)^\top H^{-1}(\partial_iF_e)
\Bigr]
=
e^\top
\Bigl(\nlsum_{i=1}^n D_iSD_i\Bigr)e.
\label{eq:Fe-dirichlet}
\end{equation}

Substituting \eqref{eq:Fe-left-side} and
\eqref{eq:Fe-dirichlet} into \eqref{eq:Fe-poincare}, we obtain
\[
e^\top(I-S^{-1})e
\le
e^\top
\Bigl(\nlsum_{i=1}^nD_iSD_i\Bigr)e.
\]
Since $e\in\R^{d-1}$ was arbitrary, it follows that
\begin{equation}
I-S^{-1}
\preceq
\nlsum_{i=1}^nD_iSD_i.
\label{eq:S-master}
\end{equation}

Now, let $L=\opnorm{S}$. Using \eqref{eq:scaled-top-variance} we have $\nlsum_iD_i^2\preceq\frac1{16}I$, and hence also $$\nlsum_iD_iSD_i
\preceq
L\nlsum_iD_i^2
\preceq
\frac{L}{16}I.$$
Therefore, evaluating \eqref{eq:S-master} on the  top unit eigenvector of $S$ gives us $1-\frac1L\le\frac{L}{16}$, or equivalently 
\begin{equation}\label{eq:L-quadratic}
L^2-16L+16\ge0.
\end{equation}
Now, since the quadratic inequality \eqref{eq:L-quadratic} factors as
\[
\bigl(L-(8-4\sqrt3)\bigr)
\bigl(L-(8+4\sqrt3)\bigr)\ge0,
\]
we have that
\[
L\le 8-4\sqrt3
\qquad\text{or}\qquad
L\ge 8+4\sqrt3.
\]
Thus \eqref{eq:L-quadratic}  leaves us with two possible regimes for $L$.
To obtain the desired upper bound, it remains to exclude the second
possibility. We do so by a continuity argument, deforming $C$ to $tC$
for $0\le t\le1$. More precisely, for $t\in[0,1]$, we replace $C$ by $tC$ and let
\[
G_t=(I-t^2C^2)^{-1},
\qquad
S_t=\E_{\nu_t}G_t,
\qquad
L_t=\norm{S_t}_{\text{op}},
\]
where $\nu_t$ denotes the corresponding tilted Gaussian measure. The
preceding argument applies verbatim with $C$ replaced by $tC$, and hence we obtain that
$1-\frac{1}{L_t}
\le
\frac{t^2L_t}{16},$ or equivalently, $t^2L_t^2-16L_t+16\ge0.$

At $t=0$, we have $G_0=I$ and therefore $S_0=I$, so that $L_0=1$.
We claim that $t\mapsto S_t$, and hence $t\mapsto L_t$, is continuous.
Indeed, recall that $S_t$ has the following form
\[
S_t
=
\frac{
\E\!\bigl[
\det(I-t^2C^2)^4
(I-t^2C^2)^{-1}
\one_{\{\norm{tC}_{\text{op}}<1\}}
\bigr]
}{
\E\!\bigl[
\det(I-t^2C^2)^4
\one_{\{\norm{tC}_{\text{op}}<1\}}
\bigr]
}.
\]
After extending both integrands by zero outside
$\{\norm{tC}_{\text{op}}<1\}$, we have the point-wise bounds
\[
0\le
\det(I-t^2C^2)^4
\one_{\{\norm{tC}_{\text{op}}<1\}}
\le1
\]
and
\[
0\preceq
\det(I-t^2C^2)^4
(I-t^2C^2)^{-1}
\one_{\{\norm{tC}_{\text{op}}<1\}}
\preceq I.
\]
Thus, dominated convergence shows that both the numerator and denominator
depend continuously on $t$, and therefore so do $S_t$ and $L_t$.

Now, note that for each $t$, the quadratic inequality $t^2L_t^2-16L_t+16\ge0$ forces $L_t$ to lie on one
of two disjoint branches 
\[
L_t\le \lambda_-(t)
\qquad\text{or}\qquad
L_t\ge \lambda_+(t),
\]
where $\lambda_\pm(t)
=
\frac{8\pm4\sqrt{4-t^2}}{t^2}.$ 
Moreover, $\lambda_-(t)\longrightarrow 1,
\lambda_+(t)\longrightarrow\infty$ as 
$t\downarrow0$.

Since $L_0=1$,  the path $L_t$
therefore starts in the lower admissible region
$L_t\le\lambda_-(t)$. Since $t\mapsto L_t$ is continuous, it cannot enter the upper admissible region
$L_t\ge\lambda_+(t)$ without first passing through the excluded interval
$\bigl(\lambda_-(t),\lambda_+(t)\bigr)$. Hence, $L_t\le\lambda_-(t)$
for all $t\in[0,1].$  In particular, at $t=1$ we get $L=L_1
\le
\lambda_-(1)
=
8-4\sqrt3$, which concludes our proof.

\subsection{Proof of Proposition \ref{prop:q-bound}}

The proof is based again on the matrix-weighted Poincar\'e inequality from
Lemma~\ref{lem:cofactor-poincare}. We apply it to a suitable test function
to control the dependence between the off-diagonal block $b$ and the
resolvent $G=(I-C^2)^{-1}$. To this end, we use the same block matrix $H$ introduced in
\eqref{eq:H-lift} in the proof of Proposition~\ref{prop:S-bound}:
\[
H(g)
=
\begin{pmatrix}
I&C(g)\\
C(g)&I
\end{pmatrix}.
\]
We will consider the test function
\[
F_b(g)
=
\begin{pmatrix}
b(g)\\
0
\end{pmatrix}.
\]
This choice is tailored so that the leading quadratic term in
matrix-weighted Poincar\'e inequality is precisely the quantity
$Q=\E_\nu q$ that we wish to control, while the Dirichlet term involves
only the coefficient vectors $u_i$. Indeed, using
\[
H^{-1}
=
\begin{pmatrix}
G&-CG\\
-CG&G
\end{pmatrix},
\]
we  obtain that $F_b^\top H^{-1}F_b
=
b^\top G b,$
and hence also $\E_\nu[F_b^\top H^{-1}F_b]
=
Q.$

We next compute the weighted mean appearing in the Poincar\'e
inequality. Since
\[
H^{-1}F_b
=
\begin{pmatrix}
Gb\\
-CGb
\end{pmatrix},
\]
and both $b$ and $C$ are odd under $g\mapsto-g$, while $G$ and the
measure $\nu$ are even, the  $Gb$ is odd. Therefore, we get that $\E_\nu[H^{-1}F_b]
=
\begin{pmatrix}
0\\
-m
\end{pmatrix}$ and $\E_\nu H^{-1}
=
\begin{pmatrix}
S&0\\
0&S
\end{pmatrix},$
where $m
=
\E_\nu[CGb].$
Hence,   the mean-correction term is given by 
\[
(\E_\nu[H^{-1}F_b])^\top
(\E_\nu H^{-1})^{-1}
(\E_\nu[H^{-1}F_b])
=
m^\top S^{-1}m.
\]

Finally, using $\partial_iF_b
=
\begin{pmatrix}
u_i\\
0
\end{pmatrix},$ for the Dirichlet term we obtain that
\begin{align*}
\E_\nu\nlsum_{i=1}^n
(\partial_iF_b)^\top H^{-1}(\partial_iF_b)
&=
\nlsum_{i=1}^n \E_\nu[u_i^\top G u_i] \\
&=
\nlsum_{i=1}^n u_i^\top S u_i \\
&=
\Tr(US).
\end{align*}
Since $F_b$ is linear in $g$, Lemma~\ref{lem:cofactor-poincare} applies (cf.\ the proof of Proposition~\ref{prop:S-bound}) and yields
\begin{equation}
Q-m^\top S^{-1}m
\le
\Tr(US).
\label{eq:Q-mean-correction}
\end{equation}

Now, to obtain a bound on $Q$, we must control the mean-correction term in
\eqref{eq:Q-mean-correction}. For any $z\in\R^{d-1}$,
Cauchy--Schwarz gives
\begin{align}
(z^\top m)^2
&=
\bigl(
\E_\nu
\langle
G^{1/2}Cz,\,
G^{1/2}b
\rangle
\bigr)^2
\nonumber\\
&\le
(
\E_\nu z^\top C^2Gz
)
(
\E_\nu b^\top Gb
)
\nonumber\\
&=
Q\,z^\top(S-I)z,
\label{eq:m-cauchy-schwarz}
\end{align}
where in the last equality we used $C^2G=G-I$. Since
\eqref{eq:m-cauchy-schwarz} holds for every $z$, it is equivalent to $mm^\top
\preceq
Q(S-I).$
Conjugating by $S^{-1/2}$ gives $S^{-1/2}mm^\top S^{-1/2}
\preceq
Q(I-S^{-1}).$
The matrix on the left has rank one, with its only nonzero eigenvalue
equal to $m^\top S^{-1}m$. Since Proposition~\ref{prop:S-bound} gives $I\preceq S\preceq\Lambda I,$
we have
\[
I-S^{-1}
\preceq
\Bigl(1-\frac1\Lambda\Bigr)I,
\]
and therefore
\begin{equation}
m^\top S^{-1}m
\le
Q\Bigl(1-\frac1\Lambda\Bigr).
\label{eq:mean-correction-bound}
\end{equation}

Combining \eqref{eq:Q-mean-correction} and
\eqref{eq:mean-correction-bound}, we obtain $\frac{Q}{\Lambda}
\le
Q-m^\top S^{-1}m
\le
\Tr(US).$
Using once more $S\preceq\Lambda I$ and
$\Tr U=\sigma$, we conclude that
\[
\frac{Q}{\Lambda}
\le
\Tr(US)
\le
\Lambda\Tr U
=
\Lambda\sigma,
\]
and thus $Q\le\Lambda^2\sigma$, which is the desired estimate.

\section{Some technical statements for the proof of Theorem \ref{thm:mesoscopic-smallball}}

\subsection{Proof of Lemma \ref{lem:tail-divided-difference}}

We will first prove the lower bound on $r_q'^{[1]}$ in \eqref{eq:tail-divided-difference-lower}.
To this end, we record the first-and second-order derivatives of $r_q$, 
\[
r_q'(t)
=
\frac{2t^{2q-1}}{1-t^2},
\qquad
r_q''(t)
=
2t^{2q-2}
\frac{(2q-1)-(2q-3)t^2}{(1-t^2)^2}.
\]
 We claim that, with
$a_0=1/64$,
\begin{equation}
r_q'^{[1]}(x,y)
\ge
\frac{1}{2(1-x)(1-y)} 
\label{eq:tail-one-sided-lower}
\end{equation}
whenever $\min\{1-x,1-y\}\le \frac{a_0}{q}.$
By symmetry we may assume that $x\ge y$. Let $\delta=1-x$ and $\varepsilon=1-y$,
so that $0<\delta\le\frac{1}{64q}$ for $\varepsilon\ge\delta$. 
We distinguish two cases. Suppose first that $\varepsilon\le\frac{1}{4q}$. Then, for every $t\in[y,x]$, we  have $t\ge1-\frac{1}{4q}.$
But since $(2q-1)-(2q-3)t^2
=
2+(2q-3)(1-t^2)
\ge2,$
and
\[
t^{2q-2}
\ge
\Bigl(1-\frac{1}{4q}\Bigr)^{2q-2}
\ge
1-\frac{2q-2}{4q}
\ge\frac12,
\]
it follows that $r_q''(t)
\ge
\frac{1}{2(1-t)^2}.$
Hence, if $x\neq y$,
\begin{align*}
r_q'^{[1]}(x,y)
&=
\frac{1}{x-y}\int_y^x r_q''(t)\,dt\\
&\ge
\frac{1}{2(x-y)}
\int_y^x\frac{dt}{(1-t)^2}\\
&=
\frac{1}{2(x-y)}
\Bigl(
\frac{1}{1-x}-\frac{1}{1-y}
\Bigr)\\
&=
\frac{1}{2(1-x)(1-y)}.
\end{align*}
Moreover, also for $x=y$ the same lower bound holds, since  $r_q'^{[1]}(x,x)
=
r_q''(x)
\ge
\frac{1}{2(1-x)^2}$.

Suppose now that $\varepsilon>\frac{1}{4q}.$
Since $\delta\le1/(64q)$, using the inequality $(1+x)^n\ge1+nx$ gives $x^{2q-1}
=
(1-\delta)^{2q-1}
\ge
1-(2q-1)\delta
\ge
\frac{31}{32}.$
Consequently, we have that
\[
r_q'(x)
=
\frac{2x^{2q-1}}{\delta(1+x)}
\ge
\frac{31}{32\delta}.
\]

On the other hand, since $\varepsilon=1-y$, we have $r_q'(y)
=
\frac{2y^{2q-1}}{\varepsilon(1+y)}.$
Moreover, if $y\ge0$, then $y^{2q-1}\le1$ and $1+y\ge1$, and hence
$r_q'(y)\le\frac{2}{\varepsilon}$.
If $y<0$, then $r_q'(y)<0$, so the same upper bound holds trivially in this case. 
 Since $\frac{\delta}{\varepsilon}
<
\frac{1/(64q)}{1/(4q)}
=
\frac1{16},$
we obtain $r_q'(x)-r_q'(y)
\ge
\frac{31}{32\delta}-\frac{2}{\varepsilon}
\ge
\frac{27}{32\delta}.$
Moreover, $x-y=\varepsilon-\delta\le\varepsilon$, and therefore
\[
r_q'^{[1]}(x,y)
=
\frac{r_q'(x)-r_q'(y)}{x-y}
\ge
\frac{27}{32\delta\varepsilon}
\ge
\frac{1}{2\delta\varepsilon},
\]
which proves \eqref{eq:tail-one-sided-lower}. But, since  $r_q'$ is odd, its divided difference satisfies $r_q'^{[1]}(-x,-y)=r_q'^{[1]}(x,y).$
Applying \eqref{eq:tail-one-sided-lower} to $(-x,-y)$ therefore gives
\begin{equation}\label{eq:rqsecond}
    r_q'^{[1]}(x,y)
\ge
\frac{1}{2(1+x)(1+y)}
\end{equation}
whenever $\min\{1+x,1+y\}\le\frac{a_0}{q}.$

Finally, combining the estimates \eqref{eq:tail-one-sided-lower} and \eqref{eq:rqsecond} yields
\[
r_q'^{[1]}(x,y)
\ge
\frac{1}{4}\Bigl(
\frac{
\one_{\{\min(1-x,\,1-y)\le a_0/q\}}
}{
(1-x)(1-y)
}
+
\frac{
\one_{\{\min(1+x,\,1+y)\le a_0/q\}}
}{
(1+x)(1+y)
}
\Bigr)
\]
which proves the first assertion \eqref{eq:tail-divided-difference-lower} with $c_0=1/4$.

We now turn to proving the upper bound \eqref{eq:tail-divided-difference-upper}. Using $\ell''(t)
=
\frac{2(1+t^2)}{(1-t^2)^2},$ we get  for every $t\in(-1,1)$,
\begin{align*}
\frac{r_q''(t)}{\ell''(t)}
&=
|t|^{2q-2}
\frac{(2q-1)-(2q-3)t^2}{1+t^2}\\
&\le
2q\,|t|^{2q-2}.
\end{align*}
Thus, $r_q''(t)
\le
2q\,|t|^{2q-2}\ell''(t).$
Finally, let $r=\max\{|x|,|y|\}$, then we have
\begin{align*}
r_q'^{[1]}(x,y)
&=
\int_0^1
r_q''\bigl((1-s)y+sx\bigr)\,ds\\
&\le
2q\,r^{2q-2}
\int_0^1
\ell''\bigl((1-s)y+sx\bigr)\,ds\\
&=
2q\,r^{2q-2}\ell'^{[1]}(x,y),
\end{align*}
which proves \eqref{eq:tail-divided-difference-upper} with $C_0=2$.

\subsection{Proof of Lemma \ref{lem:potential-hessian-lower}}

Using that $v_{q,\alpha,\beta}
=
\beta\ell+(\alpha-\beta)r_q,$ we obtain
\begin{equation}
D^2\mathcal V(B)[Y,Y]
=
\beta D^2\Tr\ell(B)[Y,Y]
+
(\alpha-\beta)D^2\Tr r_q(B)[Y,Y].
\label{eq:potential-hessian-decomposition}
\end{equation}
Moreover, using  equation \eqref{eq:log-barrier-hessian} and 
$\norm{E^\rho(Y)}_F^2
=
\norm{E_{\mathrm{bdry}}^\rho(Y)}_F^2
+
\norm{E_{\mathrm{bulk}}^\rho(Y)}_F^2$, we have that
\begin{equation}
    D^2\Tr\ell(B)[Y,Y] = \nlsum_{\rho\in\{\pm1\}}
\bigl(
\norm{E_{\mathrm{bdry}}^\rho(Y)}_F^2
+
\norm{E_{\mathrm{bulk}}^\rho(Y)}_F^2
\bigr).
\label{eq:log-hessian-boundary-bulk}
\end{equation}

Substituting \eqref{eq:tail-hessian-lower} and \eqref{eq:log-hessian-boundary-bulk} into
\eqref{eq:potential-hessian-decomposition}, and using
$\alpha\ge\beta$, yields
\begin{align*}
D^2\mathcal V(B)[Y,Y]
&\ge
\nlsum_{\rho\in\{\pm1\}}
\Bigl[
\bigl(\beta+c_0(\alpha-\beta)\bigr)
\norm{E_{\mathrm{bdry}}^\rho(Y)}_F^2
+
\beta
\norm{E_{\mathrm{bulk}}^\rho(Y)}_F^2
\Bigr].
\end{align*}
By Lemma~\ref{lem:tail-divided-difference}, we may take
$c_0=1/4$.  Hence, $\beta+c_0(\alpha-\beta)
=
\frac14\alpha+\frac34\beta
\ge
\frac14\alpha.$
Therefore, we conclude that
\[
D^2\mathcal V(B)[Y,Y]
\ge
\nlsum_{\rho\in\{\pm1\}}
\bigl(
c_0\alpha
\norm{E_{\mathrm{bdry}}^\rho(Y)}_F^2
+
\beta
\norm{E_{\mathrm{bulk}}^\rho(Y)}_F^2
\bigr),
\]
as claimed.

\subsection{Proof of Lemma \ref{lem:barrier-hessian-coercivity}}

Choose an orthonormal basis $\{e_k\}_k$ of $\mathcal H$ and write $z_i=\nlsum_k z_{i,k}e_k.$
For each $k$, we let $Y_k=\nlsum_{i=1}^n z_{i,k}M_i$
By bilinearity of the Hessian, $\nlsum_{i,j=1}^n
D^2\mathcal V(B)[M_i,M_j]\,
\langle z_i,z_j\rangle
=
\nlsum_k D^2\mathcal V(B)[Y_k,Y_k].$
Moreover, we have that $E_{\mathrm{bdry}}^\rho(Y_k)
=
\nlsum_i z_{i,k}E_{i,\mathrm{bdry}}^\rho$ and  $E_{\mathrm{bulk}}^\rho(Y_k)
=
\nlsum_i z_{i,k}E_{i,\mathrm{bulk}}^\rho,$
and therefore, $\nlsum_k
\norm{E_{\mathrm{bdry}}^\rho(Y_k)}_F^2=
\mathfrak h_\rho(z)$ and $\nlsum_k
\norm{E_{\mathrm{bulk}}^\rho(Y_k)}_F^2
=
\mathfrak u_\rho(z).$
Applying now Lemma~\ref{lem:potential-hessian-lower} to each $Y_k$ and
summing over $k$ proves the claim.

\subsection{Proof of Lemma \ref{lem:cyclic-tensor-bounds}}

Define the three-index tensor $T^C_{abc}
=
\nlsum_{i=1}^n
(C_i)_{ab}(z_i)_c,$ so that $\nlsum_{i,j=1}^n
\Tr(C_iC_j)\langle z_i,z_j\rangle
=
\nlsum_{a,b,c}(T^C_{abc})^2
=
\norm{T^C}_2^2.$
On the other hand, a direct rearrangement of indices gives
$\nlsum_{i,j=1}^n
\langle z_i,C_jC_i z_j\rangle
=
\nlsum_{a,b,c}
T^C_{abc}T^C_{bca}.$ Now, note that the map $(T_{abc})\longmapsto(T_{bca})$
is an orthogonal permutation of the Euclidean tensor coordinates, and 
therefore, by Cauchy--Schwarz, $\bigl|
\nlsum_{a,b,c}
T^C_{abc}T^C_{bca}
\bigr|
\le
\norm{T^C}_2^2,$
which proves \eqref{eq:cyclic-tensor-pure}.
The proof of \eqref{eq:cyclic-tensor-mixed} follows by a similar argument, which we omit. 

\subsection{Proof of Lemma \ref{lem:pure-bulk-variance}}

On $\operatorname{ran}(Q_\rho)$ we have $H_\rho\succeq \frac{a_0}{q}I,$
and therefore $Q_\rho H_\rho^{-1}Q_\rho
\preceq
\frac{q}{a_0}\,Q_\rho.$ Using
\[
E_{i,\mathrm{bulk}}^\rho
=
Q_\rho H_\rho^{-1/2}(\rho M_i)H_\rho^{-1/2}Q_\rho,
\]
we obtain
\begin{align}
\nlsum_{i=1}^n
(E_{i,\mathrm{bulk}}^\rho)^2
&\preceq
\frac{q}{a_0}\,
Q_\rho H_\rho^{-1/2}
\Bigl(
\nlsum_{i=1}^n M_iQ_\rho M_i
\Bigr)
H_\rho^{-1/2}Q_\rho
\nonumber\\
&\preceq
\frac{q}{a_0}\,
Q_\rho H_\rho^{-1/2}
\Bigl(
\nlsum_{i=1}^nM_i^2
\Bigr)
H_\rho^{-1/2}Q_\rho
\nonumber\\
&\preceq
\frac{q}{a_0}\eta\,
Q_\rho H_\rho^{-1}Q_\rho
\nonumber\\
&\preceq
\frac{q^2}{a_0^2}\eta I_d,
\end{align}
where we used 
$M_iQ_\rho M_i\preceq M_i^2$ since $0\preceq Q_\rho\preceq I$. This proves \eqref{eq:pure-bulk-variance} with $C=\frac{1}{a_0^2}$. 
Moreover, using the elementary inequality $|\langle u,v\rangle|
\le
\frac12\bigl(\norm{u}_2^2+\norm{v}_2^2\bigr)$
and summing over $i,j$  yields
\begin{align*}
\Bigl|
\nlsum_{i,j}
\langle
z_i,
(E_{j,\mathrm{bulk}}^\rho
E_{i,\mathrm{bulk}}^\rho\otimes I_{\mathcal{H}}) z_j
\rangle
\Bigr|
&\le
\nlsum_i
\biggl\langle
z_i,
\Bigl(
\nlsum_j
(E_{j,\mathrm{bulk}}^\rho)^2\otimes I_{\mathcal{H}}
\Bigr)z_i
\biggr\rangle\\
&\le
\frac{q^2}{a_0^2}\eta\nlsum_i\norm{z_i}_2^2,
\end{align*}
which completes the proof.

\subsection{Proof of Proposition \ref{thm:defective-two-slot}}

To prove \eqref{eq:defective-two-slot} we apply the matrix-weighted
Poincar\'e inequality of Lemma~\ref{lem:matrix-weighted-curvature} to the
weights $W_{\sigma\tau}$ and $W_\rho$ on the nonempty open convex set
$\Omega=\{g\in\R^n:\opnorm{B(g)}<1\}$. We first verify its hypotheses.
Both weights are positive definite on $\Omega$, and they are $C^\infty$
there: indeed,
\[
\mathcal V(B)
=
-\alpha\log\det(I_d-B^2)
-(\alpha-\beta)\nlsum_{k=1}^{q-1}\frac{\Tr(B^{2k})}{k},
\qquad\opnorm B<1,
\]
is real-analytic on $\{\opnorm B<1\}$, and $H_\rho^{-1}$ is rational in
$g$. To see that both weights are integrable over $\Omega$, and that the
Poincar\'e inequality applies to polynomial test fields, we record the
behaviour of the weights near $\partial\Omega$. Define
\begin{equation}
\delta(g)
=
\min_{\rho\in\{\pm1\}}
\lambda_{\min}(H_\rho(g))
=
1-\opnorm{B(g)}.
\label{eq:spectral-boundary-distance}
\end{equation}
Thus $\delta(g)>0$ on $\Omega$, while
$\delta(g)\downarrow0$ as $g$ approaches $\partial\Omega$.

We first record how the weight behaves as $\delta(g)\downarrow0$.
Recall that
\[
v_{q,\alpha,\beta}(t)
=
\beta\ell(t)+(\alpha-\beta)r_q(t),
\qquad
\ell(t)=-\log(1-t^2),
\]
and $r_q(t)
=
\ell(t)-\nlsum_{k=1}^{q-1}\frac{t^{2k}}{k}.$
Hence
\begin{equation}
v_{q,\alpha,\beta}(t)
=
-\alpha\log(1-t^2)
-
(\alpha-\beta)
\nlsum_{k=1}^{q-1}\frac{t^{2k}}{k}.
\label{eq:potential-boundary-expansion}
\end{equation}
The finite sum in \eqref{eq:potential-boundary-expansion} remains
bounded for $|t|<1$, and consequently, there is a constant
$C_{\alpha,\beta,q}<\infty$ such that
\begin{equation}
e^{-v_{q,\alpha,\beta}(t)}
\le
C_{\alpha,\beta,q}(1-t^2)^\alpha,
\qquad |t|<1.
\label{eq:scalar-potential-boundary-decay}
\end{equation}

Let $\lambda$ be an eigenvalue of $B(g)$ for which the minimum in
\eqref{eq:spectral-boundary-distance} is attained,  so that, $1+\rho\lambda=\delta(g)$ for some
$\rho\in\{\pm1\}$. 
Since $1-\lambda^2
=
(1+\rho\lambda)(1-\rho\lambda)
\le
2\delta(g),$
\eqref{eq:scalar-potential-boundary-decay} gives $e^{-v_{q,\alpha,\beta}(\lambda)}
\le
C_{\alpha,\beta,q}\delta(g)^\alpha.$
Moreover, as $v_{q,\alpha,\beta}\ge0$, we see that $e^{-v_{q,\alpha,\beta}(\lambda_j)}$ is at most one for all the remaining eigenvalues $\lambda_j$, and therefore we obtain,
\begin{equation}
e^{-\mathcal V(B(g))}
\le
C_{\alpha,\beta,q}\delta(g)^\alpha.
\label{eq:matrix-potential-boundary-decay}
\end{equation}

On the other hand, by the definition of $\delta(g)$ we have $H_\rho(g)\succeq\delta(g)I_d$.

It follows that $\norm{H_\rho^{-1}}_{\text{op}}
\le
\delta(g)^{-1},$
and hence
\begin{equation}
\norm{\mathcal K_{\sigma\tau}(g)}_{\text{op}}
=
\norm{H_\tau^{-1}}_{\text{op}}
\norm{H_\sigma^{-1}}_{\text{op}}
\le
\delta(g)^{-2}.
\label{eq:two-slot-boundary-growth}
\end{equation}
Combining \eqref{eq:matrix-potential-boundary-decay} and
\eqref{eq:two-slot-boundary-growth}, we obtain
\begin{equation}
\norm{W_{\sigma\tau}(g)}
\le
C_{\alpha,\beta,q}
e^{-\norm{g}^2/2}
\delta(g)^{\alpha-2}.
\label{eq:two-slot-weight-boundary-decay}
\end{equation}

In particular, since $\alpha\ge A_*=32>4$, we have
\begin{equation}
\norm{W_{\sigma\tau}(g)}_{\text{op}}
\le
C_{\alpha,\beta,q}e^{-\norm{g}^2/2}
\qquad\text{on }\Omega,
\label{eq:two-slot-weight-gaussian-bound}
\end{equation}
so that $\int_\Omega W_{\sigma\tau}\,dg<\infty$; the same bound holds for
$W_\rho$, since $\norm{H_\rho^{-1}}_{\text{op}}\le\delta(g)^{-1}$.
Moreover, for every polynomial field $F$,
\eqref{eq:two-slot-weight-gaussian-bound} gives
$\int_\Omega F^\top W_{\sigma\tau}F\,dg\le C_{\alpha,\beta,q}\int e^{-\norm{g}^2/2}\norm{F}_2^2\,dg<\infty$,
and likewise for $W_\rho$; in particular, $\mathcal Q_{\sigma\tau}(F)$
and $\mathcal Q_{\sigma\tau}(\partial_iF)$ are finite for every polynomial
field $F$. Finally, by Proposition~\ref{prop:two-slot-curvature}, the
curvature condition \eqref{eq:abstract-curvature-gap} holds for
$W_{\sigma\tau}$ and for $W_\rho$ with $\lambda=\frac12$. Consequently,
Lemma~\ref{lem:matrix-weighted-curvature} applies to both weights and to
every polynomial test field; since $d\gamma_n$ is a constant multiple of
$e^{-\norm{g}^2/2}dg$, all Poincar\'e inequalities used below are
instances of \eqref{eq:abstract-matrix-poincare}.

Applying Lemma~\ref{lem:matrix-weighted-curvature} to the weight $W_{\sigma\tau}$ (with $\lambda=\frac12$) we obtain 

\begin{equation}
\mathcal Q_{\sigma\tau}(F)
-
m^\top S^{-1}m
\le
C\nlsum_{i=1}^n
\mathcal Q_{\sigma\tau}(\partial_iF),
\label{eq:two-slot-gap-with-projection}
\end{equation}
where
\begin{equation}
S=\E_\mu\mathcal K_{\sigma\tau},
\qquad
m=\E_\mu[\mathcal K_{\sigma\tau} f].
\label{eq:S-m-def}
\end{equation}

In the following it will be convenient to use the abbreviations 

\[
\mathcal A=\mathcal A_{\sigma\tau},
\qquad
\mathcal K=\mathcal K_{\sigma\tau}.
\]

In the remainder, our aim will be now to estimate the weighted quadratic form
$m^\top S^{-1}m$ appearing in \eqref{eq:two-slot-gap-with-projection}. 

\medskip
\noindent

Let $S_\rho
=
\E_\mu H_\rho^{-1}.$ Since $\mathcal V$ is even, also $\mu$ is even, and using that $B(-g)=-B(g)$, we get that $\E_\mu B=0$ and hence $\E_\mu H_\rho=I_d.$
The map $H\mapsto H^{-1}$ is operator convex on the positive-definite
cone, and therefore Jensen's inequality gives
\begin{equation}
S_\rho
=
\E_\mu H_\rho^{-1}
\succeq
(\E_\mu H_\rho)^{-1}
=
I_d.
\label{eq:S-rho-lower}
\end{equation}
 
Fix $z\in\R^d$ and consider
the vector field
\[
\Phi_{\rho,z}(g)
=
(H_\rho(g)-I_d)z.
\]
By Proposition~\ref{prop:two-slot-curvature} the matrix weight $W_\rho$ has normalized curvature
bounded by $\frac{1}{2}I$, and therefore applying Lemma~\ref{lem:matrix-weighted-curvature} to the weight $W_\rho$ gives,
\begin{align}
&\E_\mu
[
\Phi_{\rho,z}^\top
H_\rho^{-1}
\Phi_{\rho,z}
]
-
(
\E_\mu[H_\rho^{-1}\Phi_{\rho,z}]
)^\top
S_\rho^{-1}
(
\E_\mu[H_\rho^{-1}\Phi_{\rho,z}]
)
\nonumber\\
&\hspace{35mm}
\le
C\nlsum_{i=1}^n
\E_\mu
[
(\partial_i\Phi_{\rho,z})^\top
H_\rho^{-1}
(\partial_i\Phi_{\rho,z})
].
\label{eq:one-slot-Poincare-test}
\end{align}
Now, since $H_\rho^{-1}(H_\rho-I_d)
=
I_d-H_\rho^{-1}, $ we get that $\E_\mu[H_\rho^{-1}\Phi_{\rho,z}]
=
(I_d-S_\rho)z.$
Moreover, using $\E_\mu H_\rho=I_d$ we obtain,
\begin{align*}
\E_\mu
[
\Phi_{\rho,z}^\top
H_\rho^{-1}
\Phi_{\rho,z}
]
&=
z^\top
\E_\mu
[
(H_\rho-I_d)H_\rho^{-1}(H_\rho-I_d)
]
z\\
&=
z^\top(S_\rho-I_d)z.
\end{align*}
  Consequently, the left-hand side
of \eqref{eq:one-slot-Poincare-test} is
\[
z^\top
[
S_\rho-I_d
-
(S_\rho-I_d)S_\rho^{-1}(S_\rho-I_d)
]z
=
z^\top(I_d-S_\rho^{-1})z.
\]
Finally, using $\partial_i\Phi_{\rho,z}
=
\rho M_i z,$
 the right-hand side of \eqref{eq:one-slot-Poincare-test} is given by $Cz^\top
\bigl(
\nlsum_{i=1}^nM_iS_\rho M_i
\bigr)z.$
Since $z$ was arbitrary, we therefore obtain
\begin{equation}
I_d-S_\rho^{-1}
\preceq
C\nlsum_{i=1}^nM_iS_\rho M_i.
\label{eq:one-slot-S-master}
\end{equation}

Let $L_\rho=\opnorm{S_\rho}$. Now, using $\nlsum_iM_i^2\preceq\eta I_d$ and \eqref{eq:one-slot-S-master}, we obtain,
\begin{equation*}
1-\frac1{L_\rho}
\le
C\eta L_\rho.
\end{equation*}

Equivalently, we have that $C\eta L_\rho^2-L_\rho+1\ge0.$
Since $q\ge2$ and $q^2\eta\le c_{\mathrm{P}}$, we have $\eta\le\frac{c_{\mathrm{P}}}{4}.$
By fixing the universal threshold $c_{\mathrm{P}}$ sufficiently
small, we may therefore assume that $4C\eta<1$, so that the quadratic
has two distinct positive roots.  We rule
out the large branch by a continuity argument.

For $s\in[0,1]$, set
\[
B_s=sB,
\qquad
H_{\rho,s}=I_d+\rho B_s,
\]
let $\mu_s$ denote the corresponding tilted measure, and define
\[
S_{\rho,s}
=
\E_{\mu_s}H_{\rho,s}^{-1},
\qquad
L_\rho(s)=\opnorm{S_{\rho,s}}.
\]
Since the coefficient matrices of $B_s$ are $sM_i$, we have that $\nlsum_i(sM_i)^2\preceq s^2\eta I_d,$
and the preceding argument gives $1-\frac1{L_\rho(s)}
\le
Cs^2\eta L_\rho(s).$
Thus, for $s>0$,
\[
L_\rho(s)\le\lambda_-(s)
\qquad\text{or}\qquad
L_\rho(s)\ge\lambda_+(s),
\]
where $\lambda_\pm(s)
=
\frac{1\pm\sqrt{1-4Cs^2\eta}}
{2Cs^2\eta}.$
Moreover,
\[
L_\rho(0)=1,
\qquad
\lambda_-(s)\to1,
\qquad
\lambda_+(s)\to\infty
\quad\text{as }s\downarrow0.
\]
The map $s\mapsto L_\rho(s)$ is continuous. Indeed, $S_{\rho,s}$ is the
ratio of $\E_{\gamma_n}[e^{-\mathcal V(sB)}H_{\rho,s}^{-1}]$ and
$\E_{\gamma_n}[e^{-\mathcal V(sB)}]$, with both integrands extended by
zero outside $\{\opnorm{sB}<1\}$. By \eqref{eq:matrix-potential-boundary-decay}
(applied to $sB$) and the bound $\norm{H_{\rho,s}^{-1}}_{\text{op}}\le(1-\opnorm{sB})^{-1}$,
the integrand of the numerator is bounded by
\[
C_{\alpha,\beta,q}\bigl(1-\opnorm{sB(g)}\bigr)^{\alpha-1}\le C_{\alpha,\beta,q},
\]
and, for fixed $g$, it is continuous in $s$ (it tends to zero as
$\opnorm{sB(g)}\uparrow1$, since $\alpha>1$); the denominator is likewise
bounded, continuous in $s$, and positive. Hence dominated convergence
yields the continuity of $s\mapsto S_{\rho,s}$, and thus of $L_\rho(s)$.
Therefore a path starting on the lower branch cannot reach the upper branch without
passing through the forbidden interval
$(\lambda_-(s),\lambda_+(s))$.  Hence $L_\rho(s)\le\lambda_-(s)$ for all $s\in[0,1]$.
At $s=1$ this gives $L_\rho
\le
\frac{2}{1+\sqrt{1-4C\eta}}
\le
1+C'\eta$, and therefore,
\[
I_d\preceq S_\rho\preceq(1+C'\eta)I_d.
\]
Now, let $T
:=
\E_\mu\mathcal A
=
\E_\mu[H_\tau\otimes H_\sigma]$. Using $\E_\mu B=0$ gives
\begin{equation}
T
=
I_{d^2}
+
\sigma\tau\,\E_\mu[B\otimes B].
\label{eq:T-expansion}
\end{equation}
We claim that the second term  in \eqref{eq:T-expansion} has operator norm at most $\eta$.
To this end, factor $\Cov_\mu(g)=UU^\top$
and define $N_k :=\nlsum_{i=1}^nU_{ik}M_i.$ Using that $\mu$ is $1$-strongly log-concave (since $\mathcal{V}$ is convex)  we have by Brascamp--Lieb $\Cov_\mu(g)\preceq I_n$, and therefore we obtain for every $x\in\R^d$, $x^\top
\bigl(\nlsum_kN_k^2\bigr)x = \nlsum_k\norm{N_kx}^2 \le
\nlsum_i\norm{M_ix}^2$, or equivalently $\nlsum_kN_k^2
\preceq
\nlsum_iM_i^2$. Moreover, using $\E_\mu[B\otimes B]
= \nlsum_kN_k\otimes N_k$, the operator Cauchy--Schwarz inequality  yields $\norm[\big]{
\nlsum_kN_k\otimes N_k
}
\le
\norm[\big]{\nlsum_kN_k^2}
\le \norm[\big]{\nlsum_iM_i^2}\le
\eta.$
Hence, we obtain, 
\begin{equation}
\norm{T-I_{d^2}}
\le
\eta.
\label{eq:T-close}
\end{equation}

We next show that $S=\E_\mu\mathcal K$ is also close to the identity.
To this end, fix $z\in\R^{d^2}$ and apply once more the matrix weighted Poincar\'e inequality 
of Lemma \ref{lem:matrix-weighted-curvature} with respect to the weight $W_{\sigma\tau}$ to the vector field $\Psi_z(g)
=
(\mathcal A(g)-I_{d^2})z$,
\begin{equation}
\mathcal Q_{\sigma\tau}(\Psi_z)
-
\E_\mu[\mathcal K \Psi_z]^\top S^{-1}\E_\mu[\mathcal K \Psi_z]
\le
C\nlsum_{i=1}^n
\mathcal Q_{\sigma\tau}(\partial_i\Psi_z),
\label{eq:psi-poincare}
\end{equation}
Since $\mathcal K=\mathcal A^{-1}$, we get $\mathcal K(\mathcal A-I_{d^2})
=
I_{d^2}-\mathcal K,$ and 
therefore $\E_\mu[\mathcal K\Psi_z]
=
(I_{d^2}-S)z.$
Furthermore,
\begin{align*}
\E_\mu
[
\Psi_z^\top\mathcal K\Psi_z
]
&=
z^\top
\E_\mu
[
(\mathcal A-I)
\mathcal A^{-1}
(\mathcal A-I)
]z\\
&=
z^\top(T-2I_{d^2}+S)z.
\end{align*}
Since, $\E_\mu[\mathcal K \Psi_z]^\top S^{-1}\E_\mu[\mathcal K \Psi_z] = z^\top
(S-I_{d^2})S^{-1}(S-I_{d^2})z,$ the left-hand side of \eqref{eq:psi-poincare} 
simplifies to $z^\top(T-S^{-1})z.$ 
It remains to estimate the corresponding Dirichlet term in \eqref{eq:psi-poincare}.  Using, $\partial_i\mathcal A
=
\tau M_i\otimes H_\sigma
+
H_\tau\otimes\sigma M_i$ and  $(U+V)\mathcal K(U+V)
\preceq
2U\mathcal KU+2V\mathcal KV$,
we obtain
\[
(\partial_i\mathcal A)
\mathcal K
(\partial_i\mathcal A)
\preceq{} 2(
M_iH_\tau^{-1}M_i
)\otimes H_\sigma+
2H_\tau\otimes
(
M_iH_\sigma^{-1}M_i
).
\]
Therefore, since  $H_\rho\preceq2I_d$ on $\Omega$ we have,
\begin{align}
\nlsum_i
\E_\mu
[
(\partial_i\mathcal A)
\mathcal K
(\partial_i\mathcal A)
]
&\preceq
C
\Bigl(
\nlsum_iM_iS_\tau M_i
\Bigr)\otimes I_d +
CI_d\otimes
\Bigl(
\nlsum_iM_iS_\sigma M_i
\Bigr)
\nonumber\\
&\preceq
C\eta I_{d^2}.
\label{eq:A-derivative-bound}
\end{align}
Finally, we conclude from  \eqref{eq:psi-poincare} and \eqref{eq:A-derivative-bound} that
\begin{equation}
0
\preceq
T-S^{-1}
\preceq
C\eta I_{d^2}.
\label{eq:T-Sinverse}
\end{equation}
Now, combining \eqref{eq:T-close} and \eqref{eq:T-Sinverse}, we  obtain
\begin{equation*}
\norm{S^{-1}-I_{d^2}}_{\text{op}}
\le
\norm{S^{-1}-T}_{\text{op}}
+
\norm{T-I_{d^2}}_{\text{op}}
\le
(C+1)\eta.
\end{equation*}
Set $\varepsilon=(C+1)\eta$.  Since $q\ge2$ and
$q^2\eta\le c_{\mathrm{P}}$, by choosing the universal threshold
$c_{\mathrm{P}}$ sufficiently small we may assume that
$\varepsilon\le1/2$.  Hence $(1-\varepsilon)I_{d^2}
\preceq
S^{-1}
\preceq
(1+\varepsilon)I_{d^2},$
and therefore
\[
\frac{1}{1+\varepsilon}I_{d^2}
\preceq
S
\preceq
\frac{1}{1-\varepsilon}I_{d^2},
\]
from which we obtain,
\begin{equation}
\norm{S-I_{d^2}}_{\text{op}}
\le
\frac{\varepsilon}{1-\varepsilon}
\le
C'\eta
\label{eq:S-two-close}
\end{equation}
for a universal constant $C'>0$.

We will now use \eqref{eq:S-two-close} to control 
$m^\top S^{-1}m$.  First, note that
\[
(\mathcal K-I_{d^2})
\mathcal K^{-1}
(\mathcal K-I_{d^2}) = \mathcal K-2I_{d^2}+\mathcal K^{-1}= \mathcal K-2I_{d^2}+\mathcal A.
\]
Then, using \eqref{eq:T-close} and
\eqref{eq:S-two-close} we obtain, 
\begin{equation}
\E_\mu
[
(\mathcal K-I_{d^2})
\mathcal K^{-1}
(\mathcal K-I_{d^2})
]
=
S-2I_{d^2}+T
\preceq
C\eta I_{d^2},
\label{eq:D-projection}
\end{equation}
Write
\[
m
=
\E_\mu f
+
r,
\qquad
r
=
\E_\mu[(\mathcal K-I_{d^2})f].
\]
For any $z\in\R^{d^2}$,  Cauchy--Schwarz then gives
\begin{align}
|z^\top r|^2
&=
\bigl|
\E_\mu
[
z^\top(\mathcal K-I_{d^2})f
]
\bigr|^2
\nonumber\\
&\le
\E_\mu[
z^\top (\mathcal K-I_{d^2})
\mathcal K^{-1}
(\mathcal K-I_{d^2})z
]
\E_\mu[f^\top\mathcal K f]
\nonumber\\
&\le
C\eta\norm{z}^2
\mathcal Q_{\sigma\tau}(F),
\label{eq:weighted-mean-CS}
\end{align}
where we used \eqref{eq:D-projection} in  the last step.
In particular, since  $z$ was arbitrary this shows $\norm{r}^2
\le
C\eta\mathcal Q_{\sigma\tau}(F).$ Using again Cauchy-Schwarz we moreover have, $\norm{\E_\mu f}^2
\le
\E_\mu\norm{f}^2
=
\E_\mu\hsnorm{F}^2$.
Now, since  \eqref{eq:S-two-close} implies $S^{-1}\preceq CI_{d^2}$, we conclude that

\begin{align}
m^\top S^{-1}m
&\le
C\norm{m}^2
\nonumber\\
&\le
C\norm{\E_\mu f}^2
+
C\norm r^2
\nonumber\\
&\le
C\E_\mu\hsnorm{F}^2
+
C\eta\mathcal Q_{\sigma\tau}(F).
\label{eq:projection-final}
\end{align}

Finally, inserting \eqref{eq:projection-final} into
\eqref{eq:two-slot-gap-with-projection} yields
\[
\mathcal Q_{\sigma\tau}(F)
\le
C\E_\mu\hsnorm{F}^2
+
C\nlsum_i\mathcal Q_{\sigma\tau}(\partial_iF)
+
C\eta\mathcal Q_{\sigma\tau}(F).
\]
Finally, since $q\ge2$ and $q^2\eta\le c_{\mathrm{P}}$, we have $C\eta
\le
\frac{C c_{\mathrm{P}}}{4}.$
We choose the universal threshold $c_{\mathrm{P}}$ sufficiently
small so that $\frac{C c_{\mathrm{P}}}{4}\le\frac12$, which completes the proof.

\subsection{Proof of Lemma \ref{lem:polynomial-derivatives}}

Let $h_1,\dots,h_j$ be independent standard Gaussian vectors in
$\R^n$, which are also independent  of $g$.  Then, we have 
\[
\norm{\nabla^jP(g)}_{\mathrm{HS}}^2
=
\E_h
\hsnorm{
D^jP(g)[h_1,\dots,h_j].
}^2.
\]

Since $P=B^{q-1}Y$ is a product of $q$ linear matrix series, repeated
application of the product rule shows that $D^jP(g)[h_1,\dots,h_j]$
is a sum of at most $q^j$ nonzero terms in which each term is a product $Z_1\cdots Z_q$
of exactly $q$ linear matrix series. Each of the matrix series appearing in these products is either a series evaluated at $g$, such as $B(g)$ or $Y(g)$, or a matrix series evaluated at
one of the independent Gaussian directions, such as $B(h_r)$ or
$Y(h_r)$.

Consider any such a product,  applying H\"older inequality we obtain, 
\begin{equation}
\E_{\mu,h}\hsnorm{Z_1\cdots Z_q}^2
\le
\nlprod_{\ell=1}^q
(
\E_{\mu,h}\Tr|Z_\ell|^{2q}
)^{1/q}.
\label{eq:probability-holder}
\end{equation}

If $Z_\ell$ is  a Gaussian series depending on $h_r$, then by
Lemma~\ref{lem:gaussian-trace-moment} we get that 
\begin{equation}
\E_h\Tr|Z_\ell|^{2q}
\le
d(Cq\eta)^q.
\label{eq:direction-factor-moment}
\end{equation}

It remains to treat those $Z_l$ depending on $g$, whose law is the tilted
measure $\mu$ rather than the original Gaussian measure.  Recall that
$\mu$ is an even convex tilt of $\gamma_n$.  Moreover, for either of
the linear matrix series $Z=B$ or $Z=Y$, the function
\[
g\longmapsto\Tr|Z(g)|^{2q}
\]
is even and convex.  Harg\'e's convex domination inequality \cite{Harge2004} therefore
gives
\[
\E_\mu\Tr|Z(g)|^{2q}
\le
\E_{\gamma_n}\Tr|Z(g)|^{2q}.
\]
Applying now Lemma~\ref{lem:gaussian-trace-moment} then  yields
\begin{equation}
\E_\mu\Tr|Z(g)|^{2q}
\le
d(Cq\eta)^q.
\label{eq:tilted-factor-moment}
\end{equation}

Combining \eqref{eq:probability-holder},
\eqref{eq:direction-factor-moment}, and
\eqref{eq:tilted-factor-moment}, we obtain for every product arising
in $D^jP$,
\begin{equation}
\E_{\mu,h}\hsnorm{Z_1\cdots Z_q}^2
\le
d(Cq\eta)^q.
\label{eq:single-product-moment}
\end{equation}

Finally, since $D^jP$ contains at most $q^j$ such products,  using \eqref{eq:single-product-moment} and the elementary matrix inequality
\[
\hsnorm[\Big]{\nlsum_{r=1}^m X_r}^2
\le
m^2\max_{1\le r\le m}\hsnorm{X_r}^2
\]
with $m\le q^j$, we  then obtain

\[
\E_\mu\norm{\nabla^jP}_{\mathrm{HS}}^2
\le
q^{2j}d(Cq\eta)^q, 
\]
which proves \eqref{eq:polynomial-derivative-bound}.

\subsection{Proof of Lemma \ref{lem:regularity}}


For $0\le t\le1$, write $w_t(g)
=
e^{-\mathcal V(B_t(g))},$
with the convention $e^{-\infty}=0$. Thus $Z_t=\E_{\gamma_n}w_t.$
We first establish pointwise bounds on $w_t$ and its first two
derivatives which are uniform up to the spectral boundary.

Recall that, for $|s|<1$,
\[
v_{q,\alpha,\beta}(s)
=
\alpha\ell(s)
-
(\alpha-\beta)
\nlsum_{k=1}^{q-1}\frac{s^{2k}}{k},
\qquad
\ell(s)=-\log(1-s^2).
\]
Hence
\begin{equation}
e^{-v_{q,\alpha,\beta}(s)}
=
(1-s^2)^\alpha
\exp\biggl(
(\alpha-\beta)
\nlsum_{k=1}^{q-1}\frac{s^{2k}}{k}
\biggr).
\label{eq:scalar-boundary-weight}
\end{equation}
Since the finite sum in \eqref{eq:scalar-boundary-weight} is bounded
on $(-1,1)$, there exists a finite constant
$C_{\alpha,\beta,q}>0$ such that
\begin{equation}
e^{-v_{q,\alpha,\beta}(s)}
\le
C_{\alpha,\beta,q}(1-s^2)^\alpha,
\qquad |s|<1.
\label{eq:scalar-boundary-weight-upper}
\end{equation}

Fix $g\in\R^n$ and $t$ such that $\opnorm{B_t(g)}<1$, and set
\[
\delta_t(g)
=
1-\opnorm{B_t(g)}.
\]
Choose an eigenvalue $\lambda$ of $B_t(g)$ satisfying
$|\lambda|=\opnorm{B_t(g)}$. Since
\[
1-\lambda^2
=
(1-|\lambda|)(1+|\lambda|)
\le
2\delta_t(g),
\]
while $v_{q,\alpha,\beta}\ge0$, we obtain from
\eqref{eq:scalar-boundary-weight-upper} that 
\begin{equation}
w_t(g)
\le
C_{\alpha,\beta,q,d}\,
\delta_t(g)^\alpha.
\label{eq:matrix-boundary-weight-upper}
\end{equation}

We next bound derivatives of the potential. For $|s|<1$,
\[
\ell'(s)=\frac{2s}{1-s^2},
\qquad
r_q'(s)=\frac{2s^{2q-1}}{1-s^2}.
\]
If $s$ belongs to the spectrum of $B_t(g)$, then
$1-|s|\ge\delta_t(g)$, and hence $|\ell'(s)|+|r_q'(s)|
\le
\frac{4}{\delta_t(g)}.$
Consequently,
\begin{equation}
|
D\mathcal V(B_t)[Y]
|
=
\bigl|
\Tr\!(
v_{q,\alpha,\beta}'(B_t)Y
)
\bigr|
\le
C_{\alpha,d}\,
\delta_t(g)^{-1}
\norm{Y(g)}_F.
\label{eq:first-potential-derivative-bound}
\end{equation}

Similarly,
\[
\ell''(s)
=
\frac{2(1+s^2)}{(1-s^2)^2},
\]
and, using the explicit expression for $r_q''$,
\[
|r_q''(s)|
\le
\frac{Cq}{(1-|s|)^2}.
\]
Thus $|v_{q,\alpha,\beta}''(s)|
\le
C_{\alpha,q}\delta_t(g)^{-2}$
for every $s$ in the convex hull of the spectrum of $B_t(g)$.
Therefore, we obtain using 
Lemma~\ref{lem:spectral-hessian} that
\begin{equation}
|
D^2\mathcal V(B_t)[Y,Y]
|
\le
C_{\alpha,q}\,
\delta_t(g)^{-2}
\norm{Y(g)}_F^2.
\label{eq:second-potential-derivative-bound}
\end{equation}

On the open set where $\opnorm{B_t(g)}<1$, we obtain from  chain rule,
\begin{equation}
\partial_t w_t(g)
=
-
D\mathcal V(B_t)[Y]\,
w_t(g)
\label{eq:pointwise-weight-first}
\end{equation}
and
\begin{equation}
\partial_t^2 w_t(g)
=
(
D\mathcal V(B_t)[Y]^2
-
D^2\mathcal V(B_t)[Y,Y]
)
w_t(g).
\label{eq:pointwise-weight-second}
\end{equation}
Combining
\eqref{eq:matrix-boundary-weight-upper},
\eqref{eq:first-potential-derivative-bound}, and
\eqref{eq:second-potential-derivative-bound}, we obtain
\begin{align}
|\partial_t w_t(g)|
&\le
C
\norm{Y(g)}_F
\delta_t(g)^{\alpha-1},
\label{eq:weight-first-bound}\\
|\partial_t^2 w_t(g)|
&\le
C
\norm{Y(g)}_F^2
\delta_t(g)^{\alpha-2}.
\label{eq:weight-second-bound}
\end{align}

For each fixed $g$, the set
\[
I_g
=
\bigl\{
t\in[0,1]:
\opnorm{B_t(g)}<1
\bigr\}
\]
is an interval, because the operator-norm unit ball is convex and
$t\mapsto B_t(g)$ is affine. At every boundary point of $I_g$ we have
$\delta_t(g)\downarrow0$. Since $\alpha>2$,
\eqref{eq:matrix-boundary-weight-upper},
\eqref{eq:weight-first-bound}, and
\eqref{eq:weight-second-bound} show that $w_t(g)$,  $\partial_tw_t(g),$ and $\partial_t^2w_t(g)$
all tend to zero as $t$ approaches such a boundary point from within
$I_g$. Hence the extension $w_t(g)=0$ outside $I_g$ is a $C^2$
function of $t$ for every fixed $g$.

It remains to justify differentiation under the Gaussian integral.
Since $Y(g)$ is linear in $g$, there is a finite constant
$C_M>0$, depending only on the coefficient matrices, such that $\norm{Y(g)}_F
\le
C_M\norm{g}_2.$
Since $0<\delta_t(g)\le1$ whenever $\opnorm{B_t(g)}<1$,
\eqref{eq:weight-first-bound} and
\eqref{eq:weight-second-bound} imply, uniformly in $t\in[0,1]$,
\[
|\partial_tw_t(g)|
\le
C\norm{g}_2,
\qquad
|\partial_t^2w_t(g)|
\le
C\norm{g}_2^2.
\]
Both dominating functions are integrable with respect to
$\gamma_n$. Dominated convergence therefore allows us to
differentiate under the expectation twice and yields

\[
Z_t'
=
-
\E_{\gamma_n}
[
D\mathcal V(B_t)[Y]
e^{-\mathcal V(B_t)}
],
\]
and
\[
Z_t''
=
\E_{\gamma_n}
\bigl[
(
D\mathcal V(B_t)[Y]^2
-
D^2\mathcal V(B_t)[Y,Y]
)
e^{-\mathcal V(B_t)}
\bigr].
\]
The same argument, together
with the pointwise continuity in $t$, shows that $Z'$ and $Z''$ are
continuous. Thus $Z\in C^2([0,1])$.

Finally, $Z_t>0$ for every $t$, since $B_t(0)=0$ and hence
$w_t$ is strictly positive on a neighborhood of $g=0$. Therefore
$F=\log Z$ also belongs to $C^2([0,1])$. Moreover, we have that
\[
F'=\frac{Z'}{Z},
\qquad
F''=\frac{Z''}{Z}-\Bigl(\frac{Z'}{Z}\Bigr)^2,
\]
which concludes the proof.

\end{document}